\documentclass[11pt,a4paper]{amsart}
\usepackage{mathpazo}  
\usepackage{amsmath}
\usepackage{amssymb}
\usepackage{stmaryrd}
\usepackage{graphicx}
\usepackage{soul}
\usepackage[all]{xy}
\usepackage{latexsym}
\usepackage{mathrsfs}
\usepackage{fancyhdr}
\usepackage{enumitem}
\usepackage[english]{babel}
\usepackage{lipsum}
\usepackage{etoolbox}
\usepackage{colonequals}
\usepackage[colorlinks=true,linkcolor=Maroon, citecolor=LimeGreen, pagebackref]{hyperref} 
\usepackage[usenames,dvipsnames]{xcolor}
\usepackage{lipsum}
\usepackage{mwe}
\usepackage{float}
 \usepackage{tikz}
\usetikzlibrary{arrows,decorations.pathmorphing,decorations.pathreplacing,positioning,shapes.geometric,shapes.misc,decorations.markings,decorations.fractals,calc,patterns}

\tikzset{>=stealth',
     cvertex/.style={circle,draw=black,inner sep=1pt,outer sep=3pt},
     vertex/.style={circle,fill=black,inner sep=1pt,outer sep=3pt},
     star/.style={circle,fill=yellow,inner sep=0.75pt,outer sep=0.75pt},
     tvertex/.style={inner sep=1pt,font=\criptsize},
     gap/.style={inner sep=0.5pt,fill=white}}

\newcommand{\arrowrl}[3][20]
{
\hspace{-5pt}
\begin{tikzpicture}
\node (A) at (0,0) {};
\node (B) at (1,0) {};
\draw[->] ($(A)+(0,0.2)$) -- node [above] {$\scriptstyle f^*$} ($(B)+(0,0.2)$);
\draw [->] ($(B)+(0,0.2)$) -- node [below] {$\scriptstyle f_*$} ($(A)+(0,0.2)$);
\end{tikzpicture}
\hspace{-5pt}
}
\newcommand{\adj}[2][20]{\arrowrl}

\usepackage[all]{xy}
 \usepackage{tikz-cd}

\usepackage[left=2.5cm, top=2.9cm, right=2.5cm, bottom=2.8cm, bindingoffset=0.5cm]{geometry}

\newcommand{\Z}{\ensuremath{\mathbb{Z}}}
\newcommand{\C}{\ensuremath{\mathbb{C}}}

\newcommand{\Ext}{\mathrm{Ext}}

\newcommand{\Hom}{\mathrm{Hom}}

\newcommand{\Gr}{\mathrm{Gr}}

\newcommand{\cA}{\mathcal{A}}

\newcommand{\cC}{\mathcal{C}}

\newcommand{\cF}{\mathcal{F}}

\newcommand{\cI}{\mathcal{I}}

\newcommand{\bC}{\mathbb{C}}

 \newcommand{\bZ}{\mathbb{Z}}

\DeclareMathOperator{\depth}{depth}

\DeclareMathOperator{\Spec}{Spec}

\newcommand{\CM}{\operatorname{{MCM}}}

\newcommand{\Modcat}{\operatorname{mod}\nolimits}

\newcommand{\grHom}{\operatorname{{grHom}}}
\newcommand{\grExt}{\operatorname{{grExt}}}

\newcommand{\mf}[1]{\ensuremath{\mathfrak{#1}}}   

\usepackage{collectbox}    

\makeatletter
\newcommand{\sqbox}{%
    \collectbox{%
        \@tempdima=\dimexpr\width-\totalheight\relax
        \ifdim\@tempdima<\z@
            \fbox{\hbox{\hspace{-.5\@tempdima}\BOXCONTENT\hspace{-.5\@tempdima}}}%
        \else
            \ht\collectedbox=\dimexpr\ht\collectedbox+.5\@tempdima\relax
            \dp\collectedbox=\dimexpr\dp\collectedbox+.5\@tempdima\relax
            \fbox{\BOXCONTENT}%
        \fi
    }%
}
\makeatother

\newtheorem{theorem}{Theorem}[section]         
\newtheorem{lemma}[theorem]{Lemma}

\newtheorem*{thma}{Theorem A}
\newtheorem*{thmb}{Theorem B}
\newtheorem*{thmc}{Theorem C}

\newtheorem{proposition}[theorem]{Proposition}
\newtheorem{corollary}[theorem]{Corollary}
\newtheorem{definition}[theorem]{Definition}

\newtheorem{remark}[theorem]{Remark}
\newtheorem{example}[theorem]{Example}

\theoremstyle{remark}

\theoremstyle{definition}
\numberwithin{equation}{section}

\newcommand{\unl}{\underline{\ell}}
\newcommand{\um}{\underline{m}}
\newcommand{\tunl}{\widetilde{\underline{\ell}}}
\newcommand{\tum}{\widetilde{\underline{m}}}
\newcommand{\Alm}{A(\unl, \um)}
\newcommand{\Blm}{B(\unl, \um)}
\newcommand{\tAlm}{A(\tunl, \tum)}
\newcommand{\tBlm}{B(\tunl, \tum)}
\newcommand{\cdim}{\mathrm{dim}_{\mathbb{C}}}

\usepackage{upgreek}
\usepackage{caption}
\usetikzlibrary{shapes,arrows,patterns,decorations.pathreplacing,shapes.geometric}
\tikzstyle{vertex}=[circle, draw, inner sep=0pt, minimum size=4pt]
\tikzstyle{bigvertex}=[circle, draw, inner sep=0pt, minimum size=10pt]
\tikzstyle{redvertex}=[diamond, draw=black, fill=red, inner sep=0pt, minimum size=5pt]
\tikzstyle{greenvertex}=[regular polygon, regular polygon sides=3, draw=black, fill=green, inner sep=0pt, minimum size=5pt]
\tikzstyle{fvertex}=[circle, draw,fill=black, inner sep=0pt, minimum size=4pt]

\title{Categories for Grassmannian cluster algebras of infinite rank}
\author{Jenny August, Man-Wai Cheung, Eleonore Faber, Sira Gratz, Sibylle Schroll}
\date{May 2019}

\address{Department of Mathematics, Aarhus University, Ny Munkegade 118, 8000 Aarhus C, Denmark}
\email{jennyaugust@math.au.dk}

\address{School of Mathematics, Kavli IPMU (WPI), UTIAS, The University of Tokyo, Japan}
\email{manwai.cheung@ipmu.jp}

\address{
School of Mathematics, University of Leeds, Leeds, LS2 9JT, UK
}

\email{e.m.faber@leeds.ac.uk}

\address{
Department of Mathematics, Aarhus University, Ny Munkegade 118, 8000 Aarhus C, Denmark
}

\email{sira@math.au.dk}

\address{Department of Mathematics, Universit\"at zu K\"oln, Weyertal 86-90, 50931 K\"oln, Germany}
\email{schroll@math.uni-koeln.de}

\date{\today}
\thanks{
J.A.~would like to thank the Max Planck Institute for Mathematics and a DNRF Chair from the Danish National Research Foundation (grant no. DNRF156) for their support.
M.C.~is supported by NSF grant DMS-1854512 and AMS Simons Travel Grants.
E.F.~was a Marie Sk{\l}odowska-Curie fellow at the University of Leeds (funded by the European Union's Horizon 2020 research and innovation programme under the Marie Sk{\l}odowska-Curie grant agreement No 789580) and is supported by EPSRC through EP/W007509/1. S.S.~is supported by  EPSRC through the Early Career Fellowship EP/P016294/1.
} 
\begin{document}

\begin{abstract}
We construct Grassmannian categories of infinite rank, providing an infinite analogue of the Grassmannian cluster categories introduced by Jensen, King, and Su. Each Grassmannian category of infinite rank is given as the category of graded maximal Cohen-Macaulay modules over a certain hypersurface singularity. We show that generically free modules of rank $1$ in a Grassmannian category of infinite rank are in bijection with the Pl\"ucker coordinates in an appropriate Grassmannian cluster algebra of infinite rank. Moreover, this bijection is structure preserving, as it relates rigidity in the category to compatibility of Pl\"ucker coordinates. Along the way, we develop a combinatorial formula to compute the dimension of the $\Ext^1$-spaces between any two generically free modules of rank $1$ in the Grassmannian category of infinite rank.
\end{abstract}

\maketitle


\maketitle

\section{Introduction}

Grassmannians are objects of great combinatorial and geometric beauty, which arise in myriad contexts. Their coordinate rings serve as a classical example in the theory of cluster algebras, whose genesis  
by Fomin and Zelevinsky \cite{FZ1} was initially motivated by total positivity in Lie theory as propagated by Lusztig, see for example \cite{Lusztig-survey}. The combinatorial structures of Grassmannians, in relation to total positivity, were first studied by Postnikov \cite{Postnikov}. Employing these combinatorial tools, Scott \cite{Scott} showed that coordinate rings of Grassmannians indeed carry a natural cluster algebra structure, which led to these objects becoming a staple in the study of cluster algebras.

Jensen, King, and Su \cite{JKS16} introduce an additive categorification of the Grassmannian cluster algebra $\mathbb{C}[\mathrm{Gr}(k,n)]$ of finite rank via $G$-equivariant maximal Cohen--Macaulay modules over the plane curve singularities $R_{(k,n)} = \mathbb{C}[x,y]/(x^k-y^{n-k})$, where $G$ is the cyclic group of order $n$ acting on $R_{(k,n)}$ in a natural way (cf.\ Section \ref{S:JKS}). They show that rank $1$ modules in $\mathrm{MCM}_G R_{(k,n)}$ are in one-to-one correspondence with Pl\"ucker coordinates in $\mathbb{C}[\mathrm{Gr}(k,n)]$, and that this bijection preserves structure: Rigidity of subcategories of rank $1$ modules is translated to compatibility of the corresponding Pl\"ucker coordinates (i.e.\ pairwise noncrossing of $k$-subsets, cf.\ Section \ref{S: finite}). An interesting aspect of this relation is that it affords a formal connection between two famous examples of a priori unrelated ADE classifications, providing a bridge between skew-symmetric cluster algebras of finite type and simple plane curve singularities. More precisely, it relates Grassmannian cluster algebras $\mathbb{C}[\mathrm{Gr}(k,n)]$ of finite type to the simple plane curve singularities $x^k = y^{n-k}$, which occur precisely in the cases $k =2$ or $n-2$ and $n \geq 4$ (type $A_{n-3}$), $k = 3$ or $n-3$ and $n = 6$ (type $D_4$), $k=3$ or $n-3$ and $n = 7$ (type $E_6$), and $k = 3$  or $n-3$ and $n = 8$ (type $E_8$). Here, the type indicates both cluster algebra type and singularity type respectively.

We extend the theory to the infinite rank setting, that is, we let $n$ go to infinity. For a fixed $k \geq 2$, a natural object to consider on the cluster algebra side is the ring
\[
 	\mathcal{A}_k = \frac{\mathbb{C}[p_I \mid I \subseteq \mathbb{Z}, |I| = k]}{\langle \text{Pl\"ucker relations} \rangle},
\]
cf.\ Section \ref{S:colimits} for details.
This is a cluster algebra of infinite rank in the sense of \cite{GG1}, and can be viewed as a colimit of Grassmannian cluster algebras $\mathbb{C}[\mathrm{Gr}(k,n)]$ in the category of rooted cluster algebras (introduced by Assem, Dupont, and Schiffler \cite{AssemDupontSchiffler}), as discussed in depth in \cite{Gratz1}. In fact, the ring $\mathcal{A}_k$ can be interpreted as the homogeneous coordinate ring of an infinite version of the Grassmannian under a generalised Pl\"ucker embedding. In the case $k=2$ this is the space of $2$-dimensional subspaces of a profinite-dimensional (topological) vector space under the Pl\"ucker embedding constructed by Groechenig in the appendix to \cite{GG1}. This point of view naturally extends to $k \geq 3$.

We construct an analogue of the Jensen, King, and Su Grassmannian cluster categories in this infinite setting: For a fixed $k \geq 2$, the Grassmannian category of infinite rank is defined to be the category of finitely generated $\mathbb{Z}$-graded maximal Cohen--Macaulay modules over the ring $R_k = \mathbb{C}[x,y]/(x^k)$, where $x$ is in degree $1$, and $y$ is in degree $-1$. From the point of view of the singularities, this is the natural category to consider\textemdash{}the singularity $x^k = 0$ is the limit of the singularities $x^k = y^{n-k}$ as $n$ goes to infinity, and the cyclic group actions yield a circle action in the limit, giving rise to the $\mathbb{Z}$-grading.

We find that this gives categorical companions embodying the combinatorics of the infinite Grassmannians. For instance, we  naturally rediscover the combinatorial description of Pl\"ucker coordinates through certain indecomposable objects of the Grassmannian category.

\begin{thma}[Theorem \ref{T:bijection}]\label{T:A}
There is a one-to-one correspondence between Pl\"ucker coordinates in $\mathcal{A}_k$ and generically free modules of rank $1$ in $\mathrm{MCM}_\mathbb{Z} R_k$.
\end{thma}

In order to prove this, we show that every generically free module of rank $1$ in $\mathrm{MCM}_\mathbb{Z} R_k$ arises as a syzygy of a finite dimensional module in $\mathrm{gr}R_k$. This allows us to reduce to the problem of classifying cofinite homogeneous ideals; we solve this problem explicitly by naturally constructing a Pl\"ucker coordinate from any such ideal.

Crucially, the correspondence from Theorem A 
is structure preserving, in the sense that it connects the concept of rigidity in $\mathrm{MCM}_\mathbb{Z} R_k$ with the concept of compatibility of Pl\"ucker coordinates, and that of noncrossing of $k$-subsets. 

\begin{thmb}
[Theorem \ref{T:compatible}]\label{T:B}
	Let $I$ and $J$ be generically free modules of rank $1$ in $\cC = \mathrm{MCM}_\mathbb{Z} R_k$ corresponding, under the bijection from Theorem 
	A,  to Pl\"ucker coordinates $p_I$ and $p_J$ respectively. Then $\Ext^1_{\cC}(I,J) = 0$ if and only if $p_I$ and $p_J$ are compatible.
\end{thmb}

Theorem B 
is a direct consequence of a general formula for the dimension of the $\Ext^1$-space between any two given generically free modules of rank $1$ that we provide in this paper. To prove this formula, we employ the combinatorial tool of staircase paths in a $(k \times k)$-grid to extract the dimension of the $\Ext^1$-space between two such modules from the crossing pattern of the associated $k$-subsets $\unl$ and $\um$ of the corresponding Pl\"ucker coordinates, cf.\ Section \ref{S:tool}. A pair of staircase paths uniquely represents the crossing patterns of $\unl$ and $\um$, and yields two significant numbers: The number  $\alpha(\unl, \um)$ of diagonals strictly above one of the paths, and the number  $\beta(\unl, \um)$ of diagonals strictly below the other, for precise details see Definition \ref{Def:numberofdiagonals}.

\begin{thmc}[Theorem \ref{T: ext dimension}]\label{T:C}
Let $I$ and $J$ be two generically free modules of rank $1$ in $\cC = \mathrm{MCM}_\mathbb{Z} R_k$ corresponding, under the bijection from Theorem A
, to the Pl\"ucker coordinates $p_I$ and $p_J$ respectively, associated to the $k$-subsets $\unl$ and $\um$. Then
\begin{align*}
\cdim(\Ext^1_{\cC}(I,J))= \alpha(\unl, \um ) + \beta(\unl, \um ) - k - | \unl \cap \um |. 
\end{align*}
\end{thmc}
It is a direct consequence of this formula that $\Ext^1$ on generically free modules of rank $1$ is symmetric in its argument. This is not a coincidence: We provide an argument to show that the full subcategory of generically free maximal Cohen-Macaulay modules is stably 2-Calabi-Yau, using a result by Iyama and Takahashi \cite{IyamaTakahashi}. We are grateful to Osamu Iyama and Michael Wemyss for suggesting this should be the case. 

These connections provide a convincing argument for the study of Grassmannian categories of infinite rank as the appropriate categorical analogue to Grassmannian cluster algebras of infinite rank. As a further illustration, let us consider the $k = 2$ case. In the case of finite rank, the types $k=2$ and $n \geq 4$ form the simplest family of Grassmannian cluster categories. In particular, the corresponding singularities are of finite Cohen--Macaulay type (i.e.\ these Grassmannian cluster categories have finitely many indecomposable objects) and exhibit Dynkin type $A$ cluster combinatorics. In the limit, as $n$ goes to $\infty$, this mild behaviour survives: The ring $R_2= \mathbb{C}[x,y]/(x^2)$ has countable Cohen--Macaulay type, and indecomposable objects in the category $\mathrm{MCM}_\mathbb{Z} \mathbb{C}[x,y]/(x^2)$ can be classified via two-element subsets of $\mathbb{Z} \cup \{\infty\}$ (or, to use a geometric Dynkin type $A_{\infty}$ model, by arcs in an $\infty$-gon with one marked accumulation point). Furthermore, this particular Grassmannian category of infinite rank has cluster tilting subcategories, which we classify in work in progress \cite{ACFGS}, recovering the classification for the one-accumulation point case by Paquette and Y{\i}ld{\i}r{\i}m \cite{PaquetteYildirim} from a different perspective. 

The infinite rank $k =2$ case has been studied extensively in recent years from different perspectives, starting with the pioneering work by Holm and J{\o}rgensen \cite{HJ-cat}. They study the finite derived category $D^f_{dg}(\mathbb{C}[y])$, where $\mathbb{C}[y]$ is viewed as a differential graded algebra with trivial differential, and $y$ in cohomological degree $-1$, which they show exhibits cluster combinatorics of type $A_\infty$. In fact, the stable category of the subcategory of $\mathrm{MCM}_\mathbb{Z}\mathbb{C}[x,y]/(x^2)$ generated by generically free modules of rank $1$ is equivalent to $D^f_{dg}(\mathbb{C}[y])$. A different viewpoint on this category is given by a special case of the combinatorial construction of discrete cluster categories of type $A_\infty$ by Igusa and Todorov \cite{ITcyclic}. 
Recent work by Paquette and Y{\i}ld{\i}r{\i}m \cite{PaquetteYildirim} constructs a completion of the discrete cluster categories of type $A_\infty$. We note that in the one-accumulation point case, this completion coincides with the stable category of our Grassmannian category of infinite rank $\mathrm{MCM}_\mathbb{Z} \mathbb{C}[x,y]/(x^2)$.

While the story is a satisfyingly conclusive one for the $k = 2$ case, we note that the $k \geq 3$ case, which we treat in this paper simultaneously, is a different matter entirely: Already in the finite rank setting (bar a handful of exceptions), we are in wild Cohen--Macaulay type. As we let $n$ go to $\infty$, this wildness, unsurprisingly, survives, and a classification of indecomposable objects in the Grassmannian categories of infinite rank for $k \geq 3$ is out of reach. It is striking that it is still possible to classify all generically free rank $1$ Cohen--Macaulay modules via the combinatorially accessible tools from Theorem A 
, and to draw a natural connection to Grassmannian combinatorics.

\subsubsection*{Acknowledgements} 
This project started from the WINART2 (Women in Noncommutative Algebra and Representation Theory) workshop, and the authors would like to thank the organisers for this wonderful opportunity. 
They also thank the London Mathematical Society (WS-1718-03), the University of Leeds, the US National Science Foundation (MS 1900575), the Association for Women in Mathematics (DMS-1500481), and the Alfred P. Sloan foundation for supporting the workshop. Further thanks go to Alastair King for his keen interest in the project and comments on the first draft.

\section{Preliminaries}

\subsection{Grassmannian Cluster Algebras}

Grassmannian cluster categories are an additive categorification of Grassmannian cluster algebras, of which this section provides an overview.

\subsubsection{The Finite Rank Case}
\label{S: finite}
Coordinate rings of flag varieties provide an interesting source of cluster algebras. An important example thereof is the Grassmannian $\Gr(k,n)$ of $k$-subspaces of $\C^n$, viewed as a projective variety via the Pl\"ucker embedding. It was shown by Scott \cite{Scott} that its homogeneous coordinate ring $\C[\Gr(k,n)]$ carries a natural cluster algebra structure, with Pl\"ucker coordinates providing a subset of cluster variables, and exchange relations coming from Pl\"ucker relations.

Consider the Grassmannian $\Gr(k,n)$ of $k$-dimensional subspaces in $\bC^n$ as a projective variety via the Pl\"ucker embedding. Its homogeneous coordinate ring is the ring
\[
	\cA_{(k,n)} = \bC[x_I \mid I \subseteq \{1, \ldots, n\}, |I| = k] \big/ \cI_P \ ,
\]
where $\cI_P$ is the ideal generated by the Pl\"ucker relations, which are described as follows: For any two subsets $J,J' \subseteq \{1, \ldots, n\}$ with $|J| = k+1$ and $|J'| = k-1$, with $J = \{j_0, \ldots, j_{k}\}$ and $j_0 < \ldots < j_{k}$ we get a Pl\"ucker relation
\begin{eqnarray}\label{E:Plucker}
	\sum_{l=0}^{k}(-1)^lx_{J' \cup \{j_l\}}x_{J \setminus \{j_l\}}.
\end{eqnarray}
We call a subset $I \subseteq \{1, \ldots, n\}$, with $|I| = k$ a {\em $k$-subset}. The variables $x_I$ labelled by $k$-subsets are called {\em Pl\"ucker coordinates}. 

Given two $k$-subsets $I,J$ we say that $I$ and $J$ {\em cross}, if there exist $i_1,i_2 \in I\setminus J$ and $j_1,j_2 \in J \setminus I$ with
\[
    i_1 < j_1 < i_2 < j_2 \; \text{or} \; j_1 < i_1 < j_2 < i_2,
\]
and $I$ and $J$ are {\em noncrossing} if they do not cross.
Two Pl\"ucker coordinates $x_I$ and $x_J$ are {\em compatible}, if the $k$-subsets $I$ and $J$ are noncrossing.

Scott \cite{Scott} has shown that $\cA_{(k,n)}$ has the structure of a cluster algebra, where the Pl\"ucker coordinates form a subset of the cluster variables, and where maximal sets of mutually compatible Pl\"ucker coordinates provide examples of clusters in $\cA_{(k,n)}$.

\subsubsection{Colimits}\label{S:colimits}

A natural way of extending the cluster combinatorics of $\cA_{(k,n)}$ to an infinite setting is by considering the ring
\[
	\cA_{k} = \bC[x_I \mid I \subseteq \bZ, |I| = k] \big/ \cI_P,
\]
where $\cI_P$ is the ideal generated by relations of the form \eqref{E:Plucker}. Note that here, the labelling $k$-subsets are subsets of $\bZ$ of size $k$. 

Indeed, the ring $\cA_k$ can be endowed with the structure of an infinite rank cluster algebra in the sense of \cite{GG1} in uncountably infinitely many ways -- it requires us to choose some initial cluster, given, for example, by a maximal set of compatible Pl\"ucker coordinates. It was shown in \cite{Gratz1} that these cluster algebras of infinite rank can be interpreted as colimits of cluster algebras of finite rank in the category of rooted cluster algebras. Indeed, for a fixed $k$ we can write it as a colimit of the cluster algebras $\cA_{(k,n)}$ with appropriate fixed initial seeds, as illustrated in \cite{GG2}.

For $k = 2$, it was shown in the appendix to \cite{GG1} that the ring $\cA_k$ can be interpreted as the homogeneous coordinate ring of an infinite version of the Grassmannian under a generalisation of the Pl\"ucker embedding \textemdash{} this infinite version can be described as the space of $2$-dimensional subspaces of a profinite-dimensional (topological) vector space (equivalently, $2$-dimensional quotients of a countably infinite-dimensional vector space). This construction naturally extends to $k \geq 3$.

\subsection{Maximal Cohen--Macaulay Modules}
Let $R$ be a commutative ring. A finitely generated module $M$ is \emph{maximal Cohen--Macaulay (=MCM)} over $R$ if $\depth(M_\mf{p})=\dim(R_{\mf{p}})$ in $R_\mf{p}$ for all $\mf{p} \in \Spec(R)$. Note that if $R$ is local, then this property can simply be stated as $\depth(M)=\dim(R)$. If $R$ is a Gorenstein commutative ring (e.g., a hypersurface), a module $M$ is maximal Cohen--Macaulay if and only if $\Ext^i_R(M,R)=0$ for $i \neq 0$, see \cite{BuchweitzMCM}.  
Note that MCM-modules are precisely the Gorenstein projectives in case $R$ is Gorenstein. 

\subsection{Grassmannian Cluster Categories of Finite Rank}
\label{S:JKS}
In \cite{JKS16}, Jensen, King, and Su introduce an additive categorification of the natural cluster algebra structure on $\cA_{(k,n)}$. In particular, their {\em Grassmannian cluster categories} are Frobenius categories, with projective-injectives corresponding to the consecutive Pl\"ucker coordinates, i.e.\ the Pl\"ucker coordinates labelled by $k$-subsets of the form $\{i,i+1, \ldots, i+k-1\}$, where we calculate modulo $n$. This extends the cluster structure of classical cluster categories with Grassmannian combinatorics to include the coefficients of the cluster algebra  $\cA_{(k,n)}$. We briefly recall their construction here. The combinatorics of these categories has been extensively studied by Baur, King and Marsh in \cite{BKM}.

Let $k \in \Z_{\geq 2}$ and $n \geq k+2$. Consider the ring $S=\C[x,y]$. The group of $n$th roots of unity 
\[
	\mu_n = \{\zeta \in \C \mid \zeta^n = 1\}
\]
acts on $S$ via
\[
	x \mapsto \zeta x; \; y \mapsto \zeta^{-1}y.
\]
Taking the quotient by the $\mu_n$ semi-invariant function $x^k - y^{n-k}$ yields the ring
\[
	R_{(k,n)} = S/(x^k-y^{n-k}).
\]
The {\em Grassmannian cluster category} is the category
\[
	\CM_{\mu_n} R_{(k,n)}
\]
of $\mu_n$-equivariant maximal Cohen-Macaulay $R_{(k,n)}$-modules. Its rank $1$ modules are in one-to-one correspondence with the Pl\"ucker coordinates of $\cA_{(k,n)}$ and under this correspondence, vanishing $\Ext^1$ between two rank $1$ modules corresponds to the corresponding Pl\"ucker coordinates being compatible, cf.~\cite[Section 5]{JKS16}. In fact, the Grassmannian cluster category $\CM_{\mu_n} R_{(k,n)}$ is stably equivalent to the category $\mathrm{Sub}Q_k$ studied by Gei\ss, Leclerc, and Schr\"oer in \cite{GLSflagvarieties}. Therefore, it has cluster tilting subcategories, and, under the above correspondence, maximal sets of compatible Pl\"ucker coordinates provide cluster tilting subcategories.

\section{Grassmannian Categories of Infinite Rank}

In this section, we introduce the construction of infinite rank versions of Grassmannian cluster categories.

\subsection{Construction}

We fix $k \in \Z_{\geq 2}$. We generalise the construction of Grassmannian cluster categories to the infinite case, by letting $n$ go to infinity. Consider the action of the multiplicative group $\mathbb{G}_m$
(playing the role taken by $\mu_n$ in the finite case) on $S=\C[x,y]$ via
\[
	x \mapsto \zeta x; \; y \mapsto \zeta^{-1}y.
\]
Now, as a semi-invariant function, we take $x^k$. We may think of this as the infinite version of the function $x^k-y^{n-k}$ as $n$ goes to infinity; topologically the neighbourhood $(y^{n-k})$ tends to $0$ as $n$ goes to infinity. This yields the Gorenstein ring 
\[
	R_k \colonequals S/(x^k)
\]
which, when we have fixed a choice of $k$, we will often simply denote by $R$.
The category we are interested in is the category
\[
	\CM_{\mathbb{G}_m} R_k
\]
of $\mathbb{G}_m$-equivariant maximal Cohen--Macaulay $R_k$-modules. gThe character group of $\mathbb{G}_m$ is the group of integers $\Z$. This yields an equivalence of categories
\[
	\Modcat_{\mathbb{G}_m}R_k \cong \mathrm{gr} R_k
\] 
 between the category $\Modcat_{\mathbb{G}_m}R_k$ of finitely generated $\mathbb{G}_m$-equivariant $R_k$-modules and the category of finitely generated $\Z$-graded $R_k$-modules $\mathrm{gr} R_k$, where $R_k$ is viewed as a $\Z$-graded ring with $x$ in degree $1$, and $y$ in degree $-1$. This induces an equivalence of categories
\[
	\CM_{\mathbb{G}_m} R_k \cong \CM_{\Z} R_k,
\] 
where $\CM_{\Z} R_k$ is the category of graded Cohen--Macaulay $R_k$-modules, with the grading given above. Note that the objects in our category are graded $\CM$-modules over $R_k$ and the morphisms are graded morphisms of degree $0$. This means that for any morphism $f: M \xrightarrow{} N$ of modules in $\CM_{\Z}R_k$, one has $f(M_i) \subseteq N_i$, where $M_i$ is the $i$-th graded piece of $M$. We call $\CM_{\Z} R_k$ the {\em Grassmannian category of type $(k,\infty)$}, or just a {\em Grassmannian category of infinite rank}, if $k$ is clear from context.

\subsection{Generically Free Modules of Rank \texorpdfstring{$1$}{1}}

Fix $k \geq 2$ and set $R=\mathbb{C}[x,y]/(x^k)$ as above.
 Define $\cF$ to be the graded total ring of fractions of $R$, i.e.\ the ring $R$ localised at all homogeneous non-zero divisors:
 \[
 	\cF = R_y = \bC[x,y^{\pm}]/(x^k).
 \]

We consider $\cF$ as a graded ring, with the grading induced by the grading of $R$. 

\begin{definition}
    A module $M \in \mathrm{gr} R$ is {\em generically free of rank $n$} if $M\otimes_R \cF$ is a graded free $\cF$-module of rank $n$.
\end{definition}

Note that $\cF \cong R_{(x)}$, where $(x)$ is the graded minimal prime ideal of $R$. 

\begin{lemma}\label{L:maximal free submodule}
	Every generically free module $M$ of rank $n$ in $\mathrm{gr} R$ has a maximal free submodule $P$ of rank $n$ such that $M/P$ is finite dimensional. 
\end{lemma}

\begin{proof}
	Take $P$ to be a maximal free submodule of $M$. This exists, since $R$ is Noetherian, and $M$ is finitely generated. First, we see that $M/P$ is also generically free, by tensoring the short exact sequence
	\[
		0 \to P \to M \to M/P \to 0.
	\]
	with $\mathcal{F}$, which yields the 
 short exact sequence
	\[
		0 \to \mathcal{F}^m \to\mathcal{F}^n \to M/P \otimes \mathcal{F} \cong (M/P)_y \to 0,
	\]
	for some $m,n \geq 0$. This sequence splits, since $\mathcal{F}$ is graded self-injective, which can be seen using Baer's criterion. Therefore $(M/P)_y$ is graded free. Next, we show that in fact $M/P \otimes \mathcal{F} \cong (M/P)_y = 0$, which implies that $M/P$ is finite dimensional and $P$ is of rank $m = n$. Indeed, if we have $(M/P)_y = 0$, then $M/P$ is annihilated by some power of $y$, and hence is isomorphic to some quotient of some power of $R$, say $(R/(y^l))^m = (\mathbb{C}[x,y]/(x^k,y^l))^m$, which is finite dimensional.
	
	To show $(M/P)_y=0$, assume as a contradiction that we have $(M/P)_y \neq 0$. Then there exists a free submodule of $M/P$: Pick a non-zero divisor $0 \neq \frac{z}{y^l} \in (M/P)_y$. Since $(M/P)_y$ is free, we have that $x^i\frac{z}{y^l} \neq 0$ for $0 \leq i < k$. It follows that $x^iz \neq 0$ for all $0 \leq i < k$, so the annihilator of $z \in M/P$ vanishes, and $z$ generates a rank $1$ free submodule $P'$ of $M/P$. 
	We get a diagram
	\[
		\xymatrix{P \ar[r] \ar@{=}[d]& Q \ar@{->>}[r] \ar@{>->}[d] & P' \ar@{>->} [d] \\
				P \ar[r] & M \ar@{->>}[r] & M/P}
	\]
	where the right-hand square is a pull-back. The top sequence splits, and we get that $Q \cong P \oplus P'$ is a free submodule of $M$, contradicting the maximality of $P$. 
\end{proof}

In the following, we denote the graded $\Hom$ by $\grHom$, and graded $\Ext^1$ by $\grExt$. Throughout, $M(j)$ denotes the graded shift of $M$, i.e.\ $M(j)_i = M_{i+j}$.

\begin{lemma}\label{L:duality}
	If $M$ is a generically free module of rank $n$ in $\mathrm{MCM}_{\bZ}R$, its dual 
	$M^* = \grHom_R(M,R)$ is also a generically free module of rank $n$ in $\mathrm{MCM}_{\bZ}R$.
\end{lemma}

\begin{proof} By \cite[Lemma 4.2.2 (iii)]{BuchweitzMCM}, the dual $M^*$ of the MCM $M$ is again MCM. Furthermore we have
	\[
		M^* \otimes \mathcal{F} = {\grHom}_R(M,R) \otimes \mathcal{F} \cong {\grHom}_{\mathcal{F}}(M \otimes \mathcal{F}, \mathcal{F}) \cong {\grHom}_{\mathcal{F}}(\mathcal{F}^n, \mathcal{F}) \cong \mathcal{F}^n,
	\]
	where the first equivalence follows from \cite[Thm.~ 7.11]{Matsumura}.
\end{proof}

\begin{proposition}\label{L:syzygy of fd}
	Every generically free module $M$ in $\mathrm{MCM}_\bZ R$ is a syzygy of a finite dimensional module in $\mathrm{gr} R$. More precisely, we have a short exact sequence of the form
\[
 	0 \to M \to \bigoplus_{i=1}^m R(-n_i) \to N \to 0,
 \]
 where $m$ is the rank of $M$ and $N$ is finite dimensional.
\end{proposition}

\begin{proof}
Assume $M$ is a generically free module of rank $m$ in $\mathrm{MCM}_\mathbb{Z}R$. Note that $m >0$, since $M$ is MCM. If $M$ is free, we are done. So assume that $M$ is not free. By Lemma \ref{L:duality}, the dual $M^*$ is also generically free, and by Lemma \ref{L:maximal free submodule}, it has a maximal free submodule $P$ of rank $m$ such that $M^*/P$ is finite dimensional. So we have $P \cong \bigoplus_{i =1}^m R(n_i)$, and $n_i \in \mathbb{Z}$.
This yields a short exact sequence
	\[
		0 \to \bigoplus_{i=1}^m R(n_i) \to M^* \to M^*/P \to 0,
	\]
and applying the graded ${\grHom}(-,R)$ 
yields the
 short exact sequence
 \begin{eqnarray}\label{E:desired sequence}
 	0 \to M \to \bigoplus_{i=1}^m R(-n_i) \to {\grExt}(M^*/P,R) \to 0,
 \end{eqnarray}
since $(M^*/P)^* = {\grHom}(M^*/P,R) = 0$ (as $M^*/P$ is finite dimensional, and thus annihilated by a power of $y$), and ${\grExt}(M^*,R) = 0$ (as $M^*$ is MCM). Note that $M^{**} \cong M$, as MCM modules over a Gorenstein ring are reflexive by \cite[Lemma 4.2.2 (iii)]{BuchweitzMCM}.
Again, since $M^*/P$ is annihilated by some power of $y$, it is a graded $R/(y^i)$-module for some $i \in \mathbb{N}$. By \cite[Corollary 3.3.7]{Weibel} ${\grExt}(M^*/P,R)$ is a graded $R/(y^i)$-module as well. Furthermore, since both $M^*/P$ and $R$ are finitely generated graded $R$-modules, so is ${\grExt}(M^*/P,R)$. To summarise, ${\grExt}(M^*/P,R)$ is a finitely generated graded $R$-module, which is annihilated by $y^i$, and hence it is finite dimensional. Thus (\ref{E:desired sequence}) is the desired sequence.

\end{proof}

\begin{proposition}\label{L:rank001}
	A graded module $I$ in $\mathrm{MCM}_{\bZ}R$ is generically free of rank $1$ if and only if $I$ is isomorphic to a graded ideal containing a power of $y$.
\end{proposition}

\begin{proof}
    If $I$ is a graded ideal of $R$ containing a power of $y$, then there is an exact sequence
    \[
    0 \to I \to R(n) \to M \to 0
    \]
    where $n \in \mathbb{Z}$, and $M$ is finite dimensional. Note that tensoring with $\cF$ is precisely localisation at $y$ and thus is exact. Thus we obtain the short exact sequence
	\[
		0 \to I \otimes_R \cF \to \cF(n) \to M \otimes_R \cF \to 0.
	\]  
	If $M \otimes_R \cF \cong M_y \neq 0$, then no power of $y$ acts trivially on $M$, and we have an infinite descending chain of ideals $M \supset yM \supset y^2M \supset \ldots$, contradicting $M$ being finite dimensional. Therefore, the last term in the sequence vanishes and so $I$ is generically free of rank 1.
	
	Now assume that $I$ is a generically free module of rank 1 in $\mathrm{MCM}_\mathbb{Z}R$. 
By Proposition \ref{L:syzygy of fd}, we have a short exact sequence
 \[
    0 \to I \to R(-i) \to N \to 0
 \]
  for some $i \in \mathbb{Z}$ and finite dimensional $N$. Therefore, $I \cong J(-i)$ for some ideal $J$, and $I$ is cofinite, and hence contains a power of $y$.

\end{proof}

\subsection{Bijection with Pl\"ucker Coordinates}
We can easily describe the graded ideals containing a power of $y$, thanks to the following.

\begin{lemma}\label{L:homogeneous ideals} Every homogeneous ideal $I \subseteq R$ can be generated by monomials.
\end{lemma}

\begin{proof}
Let $f_m=a_0 x^m + a_1 x^{m+1} y + \cdots + a_{k-m-1}x^{k-1}y^{k-1-m}$ be a homogeneous polynomial contained in $I$ and notice that we must have $m<k$ since $x^k=0$. Note also that $m$ may be negative, in which case we assume that all $a_p =0$ for $p < -m$. 

Let $p$ the smallest index such that $a_p \neq 0$. We will show by induction that $x^{k-i}y^{k-m-i} \in I$ for $i=1, \dots, k-m-p$, or in other words, the ideal generated by $f_m$ is the same as the ideal generated by $x^{m+p}y^p$, and thus $I$ is generated by monomials.

For the $i=1$ case, multiply $f_m$ by $x^sy^s$ where $s$ satisfies $m+p+s=k-1$ and notice that $s \geq 0$ by the assumption that $a_p \neq 0$. Thus, 
\[
x^sy^sf_m= a_px^{k-1}y^{k-1-m}+ x^k( \dots) = a_px^{k-1}y^{k-1-m}
\]
belongs to $I$, and hence $x^{k-1}y^{k-1-m} \in I$.

Now for the inductive step, assume $1 < i \leq k-m-p$ and $x^{k-j}y^{k-m-j} \in I$ for all $1 \leq j \leq i-1$. Multiply $f_m$ by $x^sy^s$ where $s$ satisfies $m+p+s=k-i$ and notice that $s \geq 0$ by the assumption $i \leq k-m-p$. Thus,
\begin{align*}
x^sy^sf_m &= a_px^{k-i}y^{k-i-m}+ a_{p+1}x^{k-i+1}y^{k-i-m+1} + \dots + a_{p+i-1}x^{k-1}y^{k-1-m} +x^k( \dots) \\
&= a_px^{k-i}y^{k-i-m} + a_{p+1}x^{k-i+1}y^{k-i-m+1} + \dots + a_{p+i-1}x^{k-1}y^{k-1-m}
\end{align*}
also belongs to $I$. However, by the inductive hypothesis,
\[
a_{p+1}x^{k-i+1}y^{k-i-m+1} + \dots + a_{p+i-1}x^{k-1}y^{k-1-m}
\]
belongs to $I$ and therefore so does $x^{k-i}y^{k-i-m}$ as required.
\end{proof}

Combining Proposition \ref{L:rank001} and Lemma \ref{L:homogeneous ideals}, we are able to prove the following. 

\begin{theorem}\label{T:homogeneous ideals}
	Let $I$ be in $\mathrm{MCM}_{\mathbb{Z}}R$. Then $I$ is a generically free module of rank $1$ if and only if
	\[
		I \cong (x^{k-1}, x^{k-2}y^{i_1}, x^{k-3}y^{i_2}, \ldots, xy^{i_{k-2}}, y^{i_{k-1}})(i_k) 
	\]
	for some $ 0 \leq i_1 \leq i_2 \leq \ldots \leq i_{k-1}$, and some $i_k \in \bZ$. 
\end{theorem} 

\begin{proof}
Recall from Proposition \ref{L:rank001} that a graded module $I$ in $\mathrm{MCM}_{\mathbb{Z}}R$ is generically free of rank 1 if and only if $I$ is isomorphic to a graded ideal containing a power of $y$. The ‘if’ direction follows immediately. So now suppose that $I$ is a graded module in $\mathrm{MCM}_{\mathbb{Z}}R$ which is generically free of rank 1. By Proposition \ref{L:rank001}, $I$ is isomorphic to a homogeneous ideal of $R$ containing a power of $y$ and by Lemma \ref{L:homogeneous ideals}, this ideal must be of the form 
\begin{align}
(x^{k-1}y^{i_0}, x^{k-2}y^{i_1}, x^{k-3}y^{i_2}, \dots, xy^{i_{k-2}}, y^{i_{k-1}})(i_k) \label{e: ideal}
\end{align}
where $0 \leq i_0 \leq i_1 \leq \dots \leq i_{k-1}$ and $i_k \in \mathbb{Z}$. However, notice that since $y$ is a non-zero divisor, then as graded $R$-modules, the ideal in \eqref{e: ideal}, and hence also $I$, is isomorphic to 
\begin{align*}
(x^{k-1}, x^{k-2}y^{i_1-i_0}, x^{k-3}y^{i_2-i_0}, \dots, xy^{i_{k-2}-i_0}, y^{i_{k-1}-i_0})(i_k+i_0)
\end{align*}
as required.
\end{proof}

We can depict the generically free module of of rank $1$ in $\mathrm{MCM}_{\mathbb{Z}}R$ from Theorem \ref{T:homogeneous ideals} as follows.

\begin{figure}[H]
\begin{center}
  \begin{tikzpicture}[scale=0.67]
 \draw (-10,4.5) -- (9.6, 4.5) -- (9.6,3.5) -- (-10,3.5);
    \node at (8.6,5) {\tiny{$-i_k+k-1$}};
 	\node at (8.6,4) {$x^{k-1}$};
    \node at (6.6,5) {\tiny{$-i_k+k-2$}};
 	\node at (6.6,4) {$x^{k-1}y$};
    \node at (5.15,5) {\tiny{$\cdots$}};
 	\node at (5.15, 4) {$\cdots$};
 	\node at (3.1,4) {$x^{k-1}y^{i_1+1}$};
 	\node at (3.1,3) {$x^{k-2}y^{i_1}$};
 	\node at (0,4) {$x^{k-1}y^{i_1+2}$};
 	\node at (-3.55,4) {$\cdots$};
 	\node at (-3.55,3) {$\cdots$};
 	\node at (0,3) {$x^{k-2}y^{i_1+1}$};
    \node at (3.1,5) {\tiny{$-i_k-i_1+k-2$}}; 
    \node at (0,5) {\tiny{$-i_k-i_1+k-1$}}; 
 	\node at (-3.5,1) {$xy^{i_{k-2}}$};
    \node at (-3.55,5) {\tiny{$-i_k-i_{k-2}+1$}};
 	\node at (-5.1,1) {$\cdots$};
	\node at (-5.15,5) {\tiny{$\cdots$}};
 	\node at (-1.85,3) {$\cdots$};
 	\node at (-1.85,4) {$\cdots$};
	\node at (-1.9,5) {\tiny{$\cdots$}};
 	\node at (-6.7,1) {$xy^{i_{k-1}+1}$};
 	\node at (-6.7,0) {$y^{i_{k-1}}$};
    \node at (-6.7,5) {\tiny{$-i_k-i_{k-1}$}};
    \node at (-8.7, 5) {\small{$\deg_I$}:};
    \node at (-8.7, 4) {$\cdots$};
    \node at (-8.7, 3) {$\cdots$};
    \node at (-8.7, 1) {$\cdots$};
    \node at (-8.7, 0) {$\cdots$};
 \draw (-7.7,4.5)--(-7.7,-.5);
 \draw[dotted] (-7.7, 5.5) -- (-7.7, 4.5);
 \draw (-5.7,4.5) -- (-5.7,-.5) -- (-10,-.5);
 \draw[dotted] (-5.7, 5.5) -- (-5.7, 4.5);
\draw (-4.7,4.5) -- (-4.7,.5);
 \draw[dotted](-4.7, 4.5) -- (-4.7, 5.5);
 \draw(-2.4, 4.5) -- (-2.4, .5);
 \draw[dotted](-2.4, 4.5) -- (-2.4, 5.5);
 \draw (-10,1.5) -- (-2.4,1.5) -- (-2.4,.5)--(-10,.5);
 \draw(-1.4, 4.5) -- (-1.4, 1.5);
 \draw(-2.4, 1.5) -- (-1.4, 1.5);
 \draw[dotted](-1.4, 4.5) -- (-1.4, 5.5);
 \draw (1.6,4.5)--(1.6,2.5);
 \draw[dotted] (1.6, 5.5) --(1.6, 4.5);
 	\draw (4.6,4.5)--(4.6,2.5)--(-10,2.5);
  \draw[dotted] (4.6, 5.5) -- (4.6, 4.5);
 \draw (5.6,4.5) -- (5.6,3.5);
 \draw[dotted] (5.6,4.5) -- (5.6,5.5);
 	\draw (7.6,4.5) -- (7.6,3.5);
 	\draw[dotted] (7.6,4.5) -- (7.6,5.5);
	\node at (0,2.1){$\vdots$};
	\node at (-1.85,2.1){$\vdots$};
	\node at (-5.15,2.1){$\vdots$};
    \node at (-8.7, 2.1) {$\vdots$};
	\node at (-3.5,2.1){$\vdots$};
    \node at (-6.7, 2.1) {$\vdots$};
  \end{tikzpicture}
\end{center}
\end{figure}

\begin{remark}
    Note that Theorem \ref{T:homogeneous ideals} includes the case
	\[
		(x^{k-1}, x^{k-2}y^0, x^{k-3}y^0, \ldots, xy^0,  y^0)(i_k)  \cong R(i_k) \cong (y^j)(i_k-j),
	\]
	for any $j\geq 0$, where the latter isomorphism holds as $y$ is a non-zero divisor. 
\end{remark}

    For an ideal
    \[
       I = (x^{k-1}, x^{k-2}y^{i_1}, x^{k-3}y^{i_2}, \ldots, xy^{i_{k-2}}, y^{i_{k-1}})(i_{k})
    \] 
    with $i_1 \leq i_2 < \ldots \leq i_{k-1}$, and a homogeneous element $f \in I$, we write $\deg_I(f) = \deg(f) - i_k$.
    Associated to the ideal $I$ is the $k$-subset 
	\begin{eqnarray*}
	\unl(I)	= \big(\deg_I(y^{i_{k-1}}), \deg_I(xy^{i_{k-2}}), \ldots, \deg_I(x^{k-2}y^{i_1}), \deg_I(x^{k-1})\big) = \\  \big(-i_k-i_{k-1}, -i_k-{i_{k-2}}+1, \ldots, -i_k-i_1+k-2, -i_k+k-1\big),
	\end{eqnarray*}
which we will view as a strictly increasing tuple throughout.
We now consider again the cluster algebra of infinite rank
	\[
	    \cA_k = \bC[p_{\unl} \mid \unl \subseteq \bZ, |\unl| = k]\big/\cI_P
	 \]
where we have relabelled the Pl\"ucker coordinates by $p_{\unl}$, and where $\cI_P$ is the ideal generated by the Pl\"ucker relations (\ref{E:Plucker}). For the next result, we adapt the general set-up: To a $k$-subset $\unl\colonequals (\ell_1, \dots, \ell_k)$, as always viewed as a tuple which is strictly increasing, we associate the following graded ideal:
\begin{align*}
I(\unl) \colonequals (x^{k-1}, x^{k-2}y^{i_1}, x^{k-3}y^{i_2}, \ldots, xy^{i_{k-2}}, y^{i_{k-1}})(i_k) 
\end{align*}
where $i_k = k-1-\ell_k$ and $i_{k-p}=\ell_k-\ell_p-(k-p)$. 

\begin{theorem}\label{T:bijection}
	The generically free modules of rank $1$ in $\mathrm{MCM_{\bZ}}R$ are in bijection with the Pl\"ucker coordinates in $\cA_k$.
	This bijection is given by the inverse maps
	\begin{eqnarray*}
		\{\text{generically free modules of rank $1$ in $\mathrm{MCM_{\bZ}}R$}\} & \to &  \{\text{Pl\"ucker coordinates of $\cA_k$}\}   \\
		I 
		& \mapsto & p_{\unl(I)} \\
		 I(\unl) & \mapsfrom & p_{\unl} .
	\end{eqnarray*}

\end{theorem}

\begin{proof}
    This follows immediately from Theorem \ref{T:homogeneous ideals}.
\end{proof}

\subsection{The Subcategory of Generically Free Maximal Cohen-Macaulay Modules}

We denote by $\mathrm{MCM_{\bZ}^0}R$ the full subcategory of $\mathrm{MCM_{\bZ}} R$ consisting of generically free maximal Cohen-Macaulay modules. In particular, it contains the generically free modules of rank $1$ which correspond to the Pl\"ucker coordinates of $\mathcal{A}_k$ by Theorem \ref{T:bijection}.

Note that generically free modules are closed under extensions and that they form an admissible subcategory of $\mathrm{MCM_{\bZ}} R$, so $\mathrm{MCM_{\bZ}^0}R$ is again a Frobenius category (see for example \cite{Chen12}). Thus, the stable category $\underline{\mathrm{MCM}}_{\bZ}^0R$ is a triangulated category and the goal of this section is to show that this category is 2-Calabi-Yau by applying results from Iyama and Takahashi \cite{IyamaTakahashi}.

\begin{lemma}\label{L:Gorenstein parameter}
The Gorenstein parameter of $R=\mathbb{C}[x,y]/(x^k)$, with $x$ in degree $1$ and $y$ in degree $-1$, is $k$.
\end{lemma}

Note that this agrees with the formula for the computation of the Gorenstein parameter provided in \cite[Example 4.8f]{HunekeTight} or in \cite[Examples 3.6.15]{BrunsHerzog}. Since our ring is non-trivial in both negative and positive degrees, we provide a direct computation for the peace of mind of the reader.

\begin{proof}
	Let $\alpha$ denote the Gorenstein parameter of $R$. Since $R$ has Krull-dimension $1$, we have $\grExt^1_R(\mathbb{C},R) \cong \mathbb{C}(-\alpha)$, where  $\grExt^j_R(A,B) = \bigoplus_{i \in \mathbb{Z}} \Ext^j(A,B(i))$ for graded $R$-modules $A$ and $B$.
	
	To compute $\alpha$, denote as before by $\mathcal{F}$ the graded total ring of fractions $R_y$, and consider the sequence
	\[
		R \to \mathcal{F} \to \mathcal{F}/R,
	\]
	where the first map is localisation at $y$. We first verify this is an injective resolution of $R$: Indeed, since $y$ is a non-zero divisor, the first map is injective. Furthermore, by Baer's criterion, $\mathcal{F}$ is injective over $\mathcal{F}$. Since $\mathcal{F}$ is flat over $R$, restriction of scalars sends injectives to injectives. It follows that $\mathcal{F}$ is injective over $R$. Finally, since $R$ has injective dimension $1$ as a graded module over itself, the cokernel $\mathcal{F}/R$ must be injective as well.
	
	Now apply $\grHom_R(\mathbb{C},-)$ to this resolution. Note that the socle of $\mathcal{F}/R$ is generated by $y^{-1}x^{k-1}$ (up to multiplication by a scalar this is the only element in $\mathcal{F}/R$ that gets annihilated by both $x$ and $y$), which lives in degree $k$. Since $\mathbb{C}$ must map into the socle of $\mathcal{F}/R$, it follows that $\grHom(\mathbb{C},\mathcal{F}/R) \cong \mathbb{C}(-k)$. We know that $\grExt^1_R(\mathbb{C},R)$ is one dimensional, so we must have
	\[
		\grExt^1_R(\mathbb{C},R) \cong \grHom_R(\mathbb{C},\mathcal{F}/R) \cong \mathbb{C}(-k).
	\]
	The claim follows.
\end{proof}

\begin{lemma}\label{L:Sigma squared}
	Denote by $\Sigma$ the suspension in the stable category $\underline{\mathrm{MCM}}_{\bZ}R$. Then $\Sigma^2 \cong (k)$.
\end{lemma}

\begin{proof}
The well-known equivalence between the stable category of maximal Cohen--Macaulay modules $\underline{\mathrm{MCM}}(R)$ and the category of reduced matrix factorisations $\underline{\mathrm{MF}}(S,x^k)$ (see \cite[6.1,6.3]{Eisenbud} or \cite[Theorem 7.4]{Yos}) also holds in the graded case (cf. \cite[Remark 1.8 ]{BGS}). Thus we have an exact equivalence
	\[
		\underline{\mathrm{MCM}}_{\bZ}R \cong \underline{\mathrm{MF}}_{\bZ}(S,x^k),
	\]
	where $S = \mathbb{C}[x,y]$ with $x$ in degree $1$ and $y$ in degree $-1$, and $\underline{\mathrm{MF}}_{\bZ}(S,x^k)$ denotes the homotopy category of graded matrix factorisations of $x^k$ over $S$. Indeed, if $(d_0,d_1)$ is a matrix factorisation of $x^k$, i.e.\ $d_1d_0$ is multiplication by $x^k$, and thus a degree $k$ map. Suspension on matrix factorisations is given by the shift, when viewing them as (twisted) $2$-periodic objects, so double suspension is  just the degree shift by $k$ and so $\Sigma^2$ acts as $(k)$ on objects and morphisms.
\end{proof}

\begin{proposition} \label{P: stably 2CY}
	The category $\mathrm{MCM_{\bZ}^0}R$ of generically free maximal Cohen-Macaulay modules is stably 2-Calabi-Yau.
\end{proposition}

\begin{proof}
	By \cite[Cor.~3.5]{IyamaTakahashi}, $\underline{\mathrm{MCM}}_{\bZ}^0R$ has Serre functor $S = (\alpha)$ where $\alpha$ denotes the Gorenstein parameter of $R$. By Lemma \ref{L:Gorenstein parameter} we have $\alpha = k$, and by Lemma \ref{L:Sigma squared} it follows that
	\[
		S = (k) \cong \Sigma^2. \qedhere
	\] 
\end{proof}

\section{Compatibility}

In this section, we fix $k \geq 2$ and  continue to write $R$ for the $\Z$-graded ring $\C[x,y]/(x^k)$ with $x$ in degree $1$ and $y$ in degree $-1$. We set $\cC \colonequals \mathrm{MCM}_\Z R$ and furthermore, we denote the $\Hom$ and $\Ext^1$ bifunctors in $\cC$ by $\Hom(-,-)$ and $\Ext^1(-,-)$ respectively. We will
show that for two generically free MCM modules $I$ and $J$ of rank $1$ we have $\Ext^1(I,J) = 0$, if and only if the Pl\"ucker coordinates corresponding to $I$ and $J$ are compatible, cf.\ Section \ref{S: finite}.

The key intermediate result of this section is to compute the dimension of the $\Ext^1$-spaces between generically free modules of rank $1$ in $\mathrm{MCM}_\Z R$. A formula for this is provided in Theorem \ref{T: ext dimension} and as a consequence, we deduce our following main result.
\begin{theorem}\label{T:compatible}
Let $I, J$ be two generically free rank $1$ modules in $\mathrm{MCM}_\Z R$ and $p_{\unl(I)}$ and $p_{\unl(J)}$ the corresponding Pl\"ucker coordinates. Then 
$\Ext^1(I, J) =0$ if and only if $p_{\unl(I)}$ and $p_{\unl(J)}$ are compatible. 
\end{theorem}

To prove Theorem \ref{T:compatible} we show that the $k$-subsets $\unl \colonequals \unl(I)$ and $\um \colonequals \unl(J)$ are noncrossing. 

This section is structured as follows. In Subsection \ref{S:tool} we develop a combinatorial tool to record the crossing pattern of $\unl$ and $\um$. In Section \ref{S:dimension} we provide a general formula to calculate the dimension of the $\Ext^1$-space between any two generically free modules of rank $1$, using the tool from Section \ref{S:tool}. Section \ref{S:example} provides a concrete example in the case $k=3$. Finally, in Section \ref{S:reduction} we prove Theorem \ref{T:compatible}, using reduction to the setting where $\unl$ and $\um$ are disjoint sets. 

\subsection{Combinatorial Tool}\label{S:tool}
Given two $k$-subsets $\unl$ and $\um$, we now introduce a combinatorial tool which will help us to calculate the dimension of $\Ext^1(I(\unl),I(\um))$, as well as determining whether or not the subsets are crossing.

\begin{definition}\label{D:A(l,m)}
Let $A(\unl, \um)$ (respectively $B(\unl,\um)$) be a $(k \times k)$ grid where the vertex $\Alm_{i,j}$ is filled if $\ell_i \leq m_j$ (respectively $\Blm_{i,j}$ is filled if $\ell_i < m_j$) and is empty otherwise. 
\end{definition}

\begin{example} \label{ex: k=4 grids}
Take $k=4$ and consider the subsets $\unl$ and $\um$ with 
\begin{align*}
m_1 < \ell_1 < \ell_2 =m_2 < m_3 < \ell_3  < m_4 < \ell_4.
\end{align*}
\begin{center}
\begin{tikzpicture}
\node at (-2.5,0) {$A(\unl, \um)=$};
\node at (3.5,0) {$B(\unl, \um)=$};
\node at (0,0) {
  \begin{tikzpicture}[scale=0.67]
    \foreach \x in {1,2,3,4}
	{
	\node at (0,5-\x) {$\x$};
	}
    \foreach \y in {1,2,3,4}
	{
	\node at (\y,5) {$\y$};
	}
    \foreach \x in {1,2,3,4}
    \foreach \y in {1,2,3,4}
    {
    \node[vertex] (\x,\y) at (\x,\y) {};
    }
    \fill (2,5-1) node[fvertex] {};
    \fill (3,5-1) node[fvertex] {};
    \fill (4,5-1) node[fvertex] {};
    \fill (2,5-2) node[fvertex] {};
    \fill (3,5-2) node[fvertex] {};
    \fill (4,5-2) node[fvertex] {};
    \fill (4,5-3) node[fvertex] {};
  \end{tikzpicture}};
\node at (6,0) {
  \begin{tikzpicture}[scale=0.67]
    \foreach \x in {1,2,3,4}
	{
	\node at (0,5-\x) {$\x$};
	}
    \foreach \y in {1,2,3,4}
	{
	\node at (\y,5) {$\y$};
	}
    \foreach \x in {1,2,3,4}
    \foreach \y in {1,2,3,4}
    {
    \node[vertex] (\x,\y) at (\x,\y) {};
    }
    \fill (2,5-1) node[fvertex] {};
    \fill (3,5-1) node[fvertex] {};
    \fill (4,5-1) node[fvertex] {};
    \fill (3,5-2) node[fvertex] {};
    \fill (4,5-2) node[fvertex] {};
    \fill (4,5-3) node[fvertex] {};
  \end{tikzpicture}};
\end{tikzpicture}
\end{center}
\end{example}
\begin{lemma} \label{L: equaldisjoint}
If $\unl$ and $\um$ are disjoint $k$-subsets then $\Alm = \Blm$.
\end{lemma}
\begin{proof}
Clear from Definition \ref{D:A(l,m)}.
\end{proof}
\begin{lemma} \label{lem: regions above}
If $\Alm_{i,j}$ is filled, then so is $\Alm_{p,q}$ for all $p \leq i, q \geq j$. Further, If $\Alm_{i,j}$ is empty, then so is $\Alm_{p,q}$ for all $p \geq i, q \leq j$.
\end{lemma}
\begin{proof}
Suppose $\Alm_{i,j}$ is filled, and hence by definition $\ell_i \leq m_j$. Since the $k$-subsets $\unl$ and $\um$ are strictly increasing, if $p \leq i$, then $\ell_p \leq \ell_i$ and similarly, if $q \geq j$, then $m_j \leq m_q$. Thus, for all pairs $(p,q)$ with $p \leq i, q \geq j$, 
\begin{align*}
\ell_p \leq \ell_i \leq m_j \leq m_q
\end{align*}
and hence $\Alm_{p,q}$ is also filled. The second statement is proved similarly. 
\end{proof}
In other words, Lemma \ref{lem: regions above} tells us that there is a staircase path obtained by separating the shaded and empty regions of $\Alm$. In Example \ref{ex: k=4 grids} we obtain the following path:
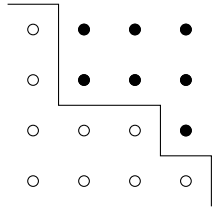
\begin{center}
  \begin{tikzpicture}[scale=0.67]
    \foreach \x in {1,2,3,4}
    \foreach \y in {1,2,3,4}
    {
    \node[vertex] at (\x,\y) {};
    }
    \fill (2,5-1) node[fvertex] {};
    \fill (3,5-1) node[fvertex] {};
    \fill (4,5-1) node[fvertex] {};
    \fill (2,5-2) node[fvertex] {};
    \fill (3,5-2) node[fvertex] {};
    \fill (4,5-2) node[fvertex] {};
    \fill (4,5-3) node[fvertex] {};
 \draw (0.5,4.5) -- (1.5,4.5)--(1.5,2.5)--(3.5,2.5)--(3.5,1.5)--(4.5,1.5)--(4.5,0.5);
  \end{tikzpicture}
\end{center}

Note that there is a completely analogous statement of Lemma \ref{lem: regions above} for $\Blm$ and thus we also get a staircase path separating the shaded and unshaded regions there. We use these staircase paths to define two nonnegative integers associated to the pair $\unl$ and $\um$ of $k$-subsets. For $1 \leq p \leq k$ we define the sets
\[
    D_p^+ = \{(i,j) \mid j-i = k-p\}
\]
to be the {\em upper diagonals of a $(k \times k)$-grid} and
\[
    D_p^- = \{(i,j) \mid i-j = k-p\}
\]
to be the {\em lower diagonals of a $(k \times k)$-grid}. Note that $D_k^+ = D_k^-$.  Below is a picture of $(4 \times 4)$-grid with the upper diagonals circled.
\begin{center}
  \begin{tikzpicture}[scale=0.67]
    \foreach \x in {1,2,3,4}
    \foreach \y in {1,2,3,4}
    {
    \node[vertex] at (\x,\y) {};
    }
\draw[rotate=45] (3.525,0) ellipse (0.4cm and 2.5cm);
\draw[rotate=45] (4.25,0) ellipse (0.3cm and 1.8cm);
\draw[rotate=45] (4.95,0) ellipse (0.25cm and 1.1cm);
\draw[rotate=45] (5.65,0) ellipse (0.2cm and 0.5cm);
\node at (5.2,0.9) {$D_4^+$};
\node at (5.2,1.9) {$D_3^+$};
\node at (5.2,2.9) {$D_2^+$};
\node at (5.2,3.9) {$D_1^+$};
  \end{tikzpicture}
\end{center}

\begin{definition}\label{Def:numberofdiagonals}
With the above notation, we introduce the following:

\begin{enumerate}
\item Let $\alpha(\unl, \um)$ be the number of upper diagonals that lie completely above the staircase path in $\Alm$ i.e.\ 
\begin{align*}
\alpha(\unl,\um) \colonequals 
\begin{cases}
\max_{1 \leq p\leq k}\{ p \mid \forall (i,j) \in D_p^+, \ \ell_i \leq m_j\} & \text{if it exists}\\
0 & \text{otherwise}.
\end{cases}
\end{align*}

\item Let $\beta(\unl, \um)$ be the number of lower diagonals that lie completely below the staircase path in $\Blm$ i.e.\ 
\begin{align*}
\beta(\unl,\um) \colonequals 
\begin{cases}
\max_{1 \leq p\leq k}\{ p \mid \forall (i,j) \in D_p^-, \ \ell_i \geq m_j\} & \text{if it exists}\\ 
0 & \text{otherwise}.   
\end{cases}
\end{align*}
\end{enumerate}
\end{definition}

When the choice of $\unl$ and $\um$ is clear we will often shorten $\alpha(\unl,\um)$ to $\alpha$ and $\beta(\unl,\um)$ to simply $\beta$.
\begin{example}
Returning to Example \ref{ex: k=4 grids}, we see that $\alpha(\unl,\um)=3$ and  $\beta(\unl,\um)=4$.
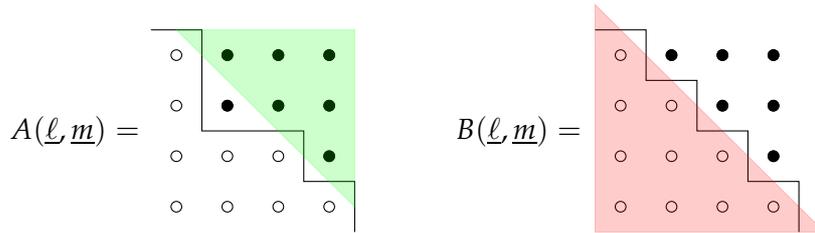
\begin{center}
  \begin{tikzpicture}[scale=0.67]
  \node at (-1,2.5) {$\Alm =$};
    \foreach \x in {1,2,3,4}
    \foreach \y in {1,2,3,4}
    {
    \node[vertex] at (\x,\y) {};
    }
    \fill (2,5-1) node[fvertex] {};
    \fill (3,5-1) node[fvertex] {};
    \fill (4,5-1) node[fvertex] {};
    \fill (2,5-2) node[fvertex] {};
    \fill (3,5-2) node[fvertex] {};
    \fill (4,5-2) node[fvertex] {};
    \fill (4,5-3) node[fvertex] {};
 \draw (0.5,4.5) -- (1.5,4.5)--(1.5,2.5)--(3.5,2.5)--(3.5,1.5)--(4.5,1.5)--(4.5,0.5);
\draw [draw=green, fill=green, opacity=0.2] (1,4.5) -- (4.5,4.5) -- (4.5,1) -- cycle;
  \end{tikzpicture}
\hspace{1cm}
  \begin{tikzpicture}[scale=0.67]
    \node at (-1,2.5) {$\Blm =$};
    \foreach \x in {1,2,3,4}
    \foreach \y in {1,2,3,4}
    {
    \node[vertex] at (\x,\y) {};
    }
    \fill (2,5-1) node[fvertex] {};
    \fill (3,5-1) node[fvertex] {};
    \fill (4,5-1) node[fvertex] {};
    \fill (3,5-2) node[fvertex] {};
    \fill (4,5-2) node[fvertex] {};
    \fill (4,5-3) node[fvertex] {};
 \draw (0.5,4.5) -- (1.5,4.5)--(1.5,3.5)--(2.5,3.5)--(2.5,2.5)--(3.5,2.5)--(3.5,1.5)--(4.5,1.5)--(4.5,0.5);
\draw [draw=red, fill=red, opacity=0.2] (0.5,0.5) -- (0.5,5) -- (5,0.5) -- cycle;
  \end{tikzpicture}
\end{center}
\end{example}

\begin{lemma} \label{L:symmetry}
    If $\unl$ and $\um$ are $k$-subsets, then $\alpha(\unl,\um) = \beta(\um, \unl)$.
\end{lemma}

\begin{proof}
    By definition $D_p^+$ lies completely above the staircase path in $\Alm$ if, for all $(i,j)$ such that $j-i=k-p$, we have $l_i \leq m_j$. Similarly, $D_p^-$ lies completely below the staircase path in $B(\um,\unl)$ if, for all $(j,i)$ with $j-i=k-p$ we have $m_j \geq l_i$. Since these conditions are the same, the number $\alpha(\unl, \um)$ of upper diagonals that lie completely above the staircase path in $\Alm$ is the same as the number $\beta(\um,\unl)$ of lower diagonals that lie completely below the staircase path in $B(\um,\unl)$. 
\end{proof}

\begin{lemma} \label{lem: single step}
If $\unl$ and $\um$ are disjoint $k$-subsets, then they are noncrossing if and only if the staircase path consists of a single step i.e.\ looks like one of the following:
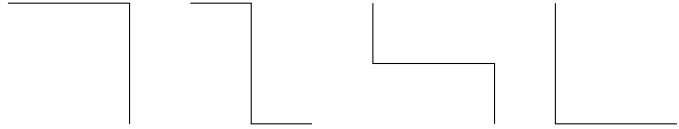
\begin{center}
\begin{tikzpicture}[scale=0.4]
 \draw (0.5,4.5) -- (4.5,4.5)--(4.5,0.5);
 \draw (6.5,4.5) -- (8.5,4.5)--(8.5,0.5)--(10.5,0.5);
 \draw (12.5,4.5) -- (12.5,2.5)--(16.5,2.5)--(16.5,0.5);
 \draw (18.5,4.5) -- (18.5,0.5)--(22.5,0.5);
\end{tikzpicture}
\end{center}
Moreover, this holds if and only if $\alpha(\unl,\um)+\beta(\unl,\um)=k$.
\end{lemma}
\begin{proof}
Since $\unl$ and $\um$ are disjoint, $\Alm=\Blm$ by Lemma \ref{L: equaldisjoint} and so there is only one staircase path associated to the pair. It is straightforward to check that if the path is one of the four given, then $\unl$ and $\um$ are noncrossing. On the other hand, if the staircase path for $\unl$ and $\um$ is not one of the four given, the path must have one of the following local configurations:
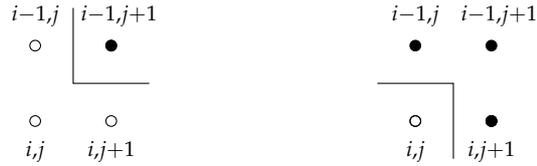
\begin{center}
  \begin{tikzpicture}[scale=1]
    \foreach \x in {1,2}
    \foreach \y in {1,2}
    {
    \node[vertex] at (\x,\y) {};
    }
    \fill (2,2) node[fvertex] {};
 \draw (1.5,2.5)--(1.5,1.5)--(2.5,1.5);
    \node at (1,0.6) {$\scriptstyle i,j$};
    \node at (1,2.4) {$\scriptstyle i-1,j$};        
    \node at (2,0.6) {$\scriptstyle i,j+1$};
    \node at (2.1,2.4) {$\scriptstyle i-1,j+1$};
    \foreach \x in {6,7}
    \foreach \y in {1,2}
    {
    \node[vertex] at (\x,) {};
    }
    \fill (7,2) node[fvertex] {};
    \fill (6,2) node[fvertex] {};
    \fill (7,1) node[fvertex] {};
 \draw (5.5,1.5) -- (6.5,1.5)--(6.5,0.5);
    \node at (6,0.6) {$\scriptstyle i,j$};
    \node at (6,2.4) {$\scriptstyle i-1,j$};        
    \node at (7,0.6) {$\scriptstyle i,j+1$};
    \node at (7.1,2.4) {$\scriptstyle i-1,j+1$};    
  \end{tikzpicture}
\end{center}
In the former case, since $\unl$ and $\um$ are disjoint, we have $m_j < l_{i-1} < m_{j+1} < l_i$ and thus $\unl$ and $\um$ are crossing. The latter case follows similarly.

For the second statement, first assume that the step path in $\Alm$ (and hence also in $\Blm$ since $\unl$ and $\um$ are disjoint) is one of the four given cases. In each case, it is easy to read off $\alpha$ and $\beta$:
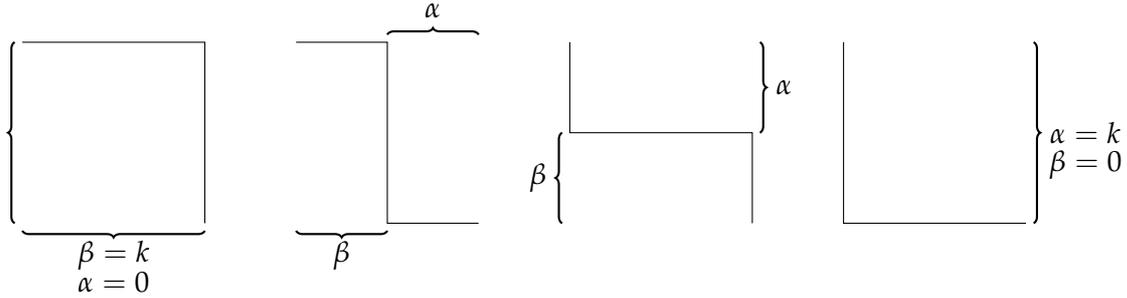
 \begin{center}
\begin{tikzpicture}[scale=0.6]
 \draw (0.5,4.5) -- (4.5,4.5)--(4.5,0.5);
\draw [thick, decoration={brace, mirror, raise=0.1cm}, decorate] (0.5,0.5)--(4.5,0.5);
\draw [thick, decoration={brace, raise=0.1cm}, decorate] (0.5,0.5)--(0.5,4.5);
\node at (2.5,-0.2) {$\beta=k$};
\node at (2.5,-0.8) {$\alpha=0$};
 \draw (6.5,4.5) -- (8.5,4.5)--(8.5,0.5)--(10.5,0.5);
\draw [thick, decoration={brace, mirror, raise=0.1cm}, decorate] (6.5,0.5)--(8.5,0.5);
\draw [thick, decoration={brace, raise=0.1cm}, decorate] (8.5,4.5)--(10.5,4.5);
\node at (7.5,-0.2) {$\beta$};
\node at (9.5,5.2) {$\alpha$};
 \draw (12.5,4.5) -- (12.5,2.5)--(16.5,2.5)--(16.5,0.5);
\draw [thick, decoration={brace, raise=0.1cm}, decorate] (12.5,0.5)--(12.5,2.5);
\draw [thick, decoration={brace, mirror, raise=0.1cm}, decorate] (16.5,2.5)--(16.5,4.5);
\node at (11.8,1.5) {$\beta$};
\node at (17.2,3.5) {$\alpha$};
 \draw (18.5,4.5) -- (18.5,0.5)--(22.5,0.5);
\draw [thick, decoration={brace, mirror, raise=0.1cm}, decorate] (22.5,0.5)--(22.5,4.5);
\node at (23.8,2.5) {$\alpha=k$};
\node at (23.8,1.8) {$\beta=0$};
\end{tikzpicture}
\end{center}
It is clear that in all cases $\alpha+\beta=k$. \\

Now assume that $\alpha+\beta=k$. Since $\unl$ and $\um$ are disjoint, we have $\Alm=\Blm$ and so both $\alpha$ and $\beta$ can be determined by looking solely at $\Alm$. 

Start by considering the case when $\beta=k$ and $\alpha=0$. Since $\alpha=0$, we know that $\Alm_{1,k}$ lies below the staircase, and thus by Lemma \ref{lem: regions above}, $\Alm_{p,q}$ lies below the staircase for all $1 \leq p,q \leq k$. Then the staircase must be:
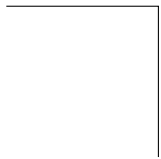
\begin{center}
\begin{tikzpicture}[scale=0.5]
 \draw (0.5,4.5) -- (4.5,4.5)--(4.5,0.5);
\end{tikzpicture}
\end{center}
Similarly, considering the case when $\alpha=k$ and $\beta=0$, this implies $\Alm_{k,1}$ lies above the staircase otherwise $\beta$  would be at least one. By Lemma \ref{lem: regions above}, this implies further that $\Alm_{p,q}$ lies above the staircase for all $1 \leq p,q \leq k$, and thus the staircase must be:
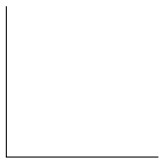
\begin{center}
\begin{tikzpicture}[scale=0.5]
 \draw (18.5,4.5) -- (18.5,0.5)--(22.5,0.5);
\end{tikzpicture}
\end{center}

Now assume that $1 < \beta < k$. Since $\beta$ is chosen to be maximal, we know that there must exist a pair $(i,j) \in D^-_{\beta+1}$  
such that $\Alm_{i,j}$ is above the staircase (i.e.\ one of the red diamond vertices in Figure 1 must be above the staircase or $\beta$ would be at least one larger). Note that such an $(i,j)$ has the form $(i, i-(k-\beta)+1)=(i,i-\alpha+1)$ where $i \in \{\alpha, \dots, k\}$.   
Similarly, since $1< \alpha < k$ is also chosen to be maximal, there exists a pair $(p,\beta+p-1) \in D^+_{\alpha+1}$ with $p \in \{1, \dots, \alpha+1\}$ such that $\Alm_{p,q}$ lies below the staircase (i.e\ one of the green triangle vertices in Figure 1).
\begin{center}
\begin{figure}[h!]
\captionsetup{width=\linewidth}
  \begin{tikzpicture}[scale=0.67]
        \node[vertex] at (1,1) {};
    \node[vertex] at (1,2) {};
    \node[vertex] at (1,3) {};
    \node[vertex] at (1,5) {};
    \node[vertex] at (2,1) {};
    \node[vertex] at (2,2) {};
    \node[vertex] at (2,4) {};
    \node[vertex] at (2,5) {};
    \node[vertex] at (3,1) {};
    \node[vertex] at (3,3) {};
    \node[vertex] at (3,4) {};
    \node[vertex] at (4,2) {};
    \node[vertex] at (4,3) {};
    \node[vertex] at (4,5) {};
    \node[vertex] at (5,1) {};
    \node[vertex] at (5,2) {};
    \node[vertex] at (5,4) {};
    \node[vertex] at (5,5) {};
\draw [draw=red, fill=red, opacity=0.2]
       (0.5,0.5) -- (0.5,4) -- (4,0.5) -- cycle;
\draw [draw=green, fill=green, opacity=0.2]
       (3,5.5) -- (5.5,5.5) -- (5.5,3) -- cycle;
    \fill (1,4) node[redvertex] {};
    \fill (2,3) node[redvertex] {};
    \fill (3,2) node[redvertex] {};
    \fill (4,1) node[redvertex] {};
    \fill (3,5) node[greenvertex] {};
    \fill (4,4) node[greenvertex] {};
    \fill (5,3) node[greenvertex] {};
\draw [thick, decoration={brace, mirror, raise=0.1cm}, decorate] (0.5,0.5)--(3.5,0.5);
\node at (2,-0.2) {$\beta$};
\draw [thick, decoration={brace, raise=0.1cm}, decorate] (3.5,5.5)--(5.5,5.5);
\node at (4.5,6.2) {$\alpha$};
  \end{tikzpicture}
\hspace{0.5cm}
  \begin{tikzpicture}[scale=0.67]
    \node[vertex] at (1,1) {};
    \node[vertex] at (1,2) {};
    \node[vertex] at (1,3) {};
    \node[vertex] at (1,5) {};
    \node[vertex] at (2,1) {};
    \node[vertex] at (2,2) {};
    \node[vertex] at (2,4) {};
    \node[vertex] at (2,5) {};
    \node[vertex] at (3,1) {};
    \node[vertex] at (3,3) {};
    \node[vertex] at (3,4) {};
    \node[vertex] at (4,2) {};
    \node[vertex] at (4,3) {};
    \node[vertex] at (4,5) {};
    \node[vertex] at (5,1) {};
    \node[vertex] at (5,2) {};
    \node[vertex] at (5,4) {};
    \node[vertex] at (5,5) {};
\node[bigvertex] at (2,3) {};
 \draw[draw=yellow, fill=yellow, opacity=0.2] (1.5,5.5) -- (1.5,2.5)--(5.5,2.5)--(5.5,5.5) -- cycle;
\draw [draw=red, fill=red, opacity=0.2]
       (0.5,0.5) -- (0.5,4) -- (4,0.5) -- cycle;
\draw [draw=green, fill=green, opacity=0.2]
       (3,5.5) -- (5.5,5.5) -- (5.5,3) -- cycle;
    \fill (1,4) node[redvertex] {};
    \fill (2,3) node[redvertex] {};
    \fill (3,2) node[redvertex] {};
    \fill (4,1) node[redvertex] {};
    \fill (3,5) node[greenvertex] {};
    \fill (4,4) node[greenvertex] {};
    \fill (5,3) node[greenvertex] {};
\draw [thick, decoration={brace, mirror, raise=0.1cm}, decorate] (0.5,0.5)--(3.5,0.5);
\node at (2,-0.2) {$\beta$};
\draw [thick, decoration={brace, raise=0.1cm}, decorate] (3.5,5.5)--(5.5,5.5);
\node at (4.5,6.2) {$\alpha$};
  \end{tikzpicture}
\hspace{0.5cm}
  \begin{tikzpicture}[scale=0.67]
        \node[vertex] at (1,1) {};
    \node[vertex] at (1,2) {};
    \node[vertex] at (1,3) {};
    \node[vertex] at (1,5) {};
    \node[vertex] at (2,1) {};
    \node[vertex] at (2,2) {};
    \node[vertex] at (2,4) {};
    \node[vertex] at (2,5) {};
    \node[vertex] at (3,1) {};
    \node[vertex] at (3,3) {};
    \node[vertex] at (3,4) {};
    \node[vertex] at (4,2) {};
    \node[vertex] at (4,3) {};
    \node[vertex] at (4,5) {};
    \node[vertex] at (5,1) {};
    \node[vertex] at (5,2) {};
    \node[vertex] at (5,4) {};
    \node[vertex] at (5,5) {};
\node[bigvertex] at (4,1) {};
 \draw[draw=yellow, fill=yellow, opacity=0.2] (3.5,5.5) -- (3.5,0.5)--(5.5,0.5)--(5.5,5.5) -- cycle;
\draw [draw=red, fill=red, opacity=0.2] (0.5,0.5) -- (0.5,4) -- (4,0.5) -- cycle;
\draw [draw=green, fill=green, opacity=0.2] (3,5.5) -- (5.5,5.5) -- (5.5,3) -- cycle;
    \fill (1,4) node[redvertex] {};
    \fill (2,3) node[redvertex] {};
    \fill (3,2) node[redvertex] {};
    \fill (4,1) node[redvertex] {};
    \fill (3,5) node[greenvertex] {};
    \fill (4,4) node[greenvertex] {};
    \fill (5,3) node[greenvertex] {};
\draw [thick, decoration={brace, mirror, raise=0.1cm}, decorate] (0.5,0.5)--(3.5,0.5);
\node at (2,-0.2) {$\beta$};
\draw [thick, decoration={brace, raise=0.1cm}, decorate] (3.5,5.5)--(5.5,5.5);
\node at (4.5,6.2) {$\alpha$};
\draw (0.5,5.5)--(3.5,5.5)--(3.5,0.5)--(5.5,0.5);
  \end{tikzpicture}
\caption{With $k=5$, $\beta=3$ and $\alpha=2$, one of the red diamond vertices must lie above the staircase (not drawn), and one of the green triangle vertices must lie below. If one of the inner diamond vertices lies above the staircase, the second picture shows all the triangle vertices must also lie above. The last picture shows choosing one of the outer ones leaves one triangle vertex that may lie below.}
\end{figure}
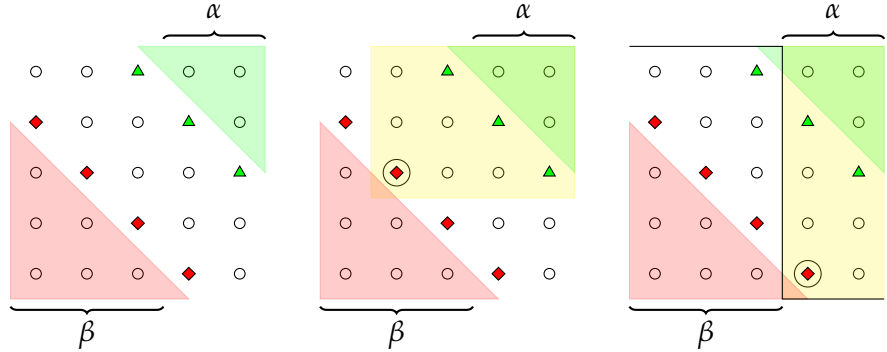
\end{center}
Suppose $\Alm_{i,i-\alpha+1}$ lies above the staircase for some $i \in \{\alpha+1, \dots, k-1\}$ (i.e.\ one of the inner vertices on the diagonal $D_{\beta+1}^-$). 
Then for all $p \in \{1, \dots, \alpha+1\}$, we have $p \leq i$ and 
\begin{align*}
i-\alpha+1 \leq (k-1)-\alpha+1= k-\alpha = \beta \leq \beta+p-1 
\end{align*}
and hence by Lemma \ref{lem: regions above}, all the points $\Alm_{p, \beta+p-1}$ with $p \in \{1, \dots, \alpha+1\}$ (all the triangle vertices in Figure 1) lie above the staircase. Or equivalently, the entire diagonal $D_{\alpha+1}^+$ lies above the staircase, contradicting the maximality of $\alpha$. Therefore, we must have either $\Alm_{\alpha,1}$ 
or $\Alm_{k,\beta+1}$ lying above the staircase.\\

If $\Alm_{k,\beta+1}$ lies above the staircase then, similar to above, all the points $\Alm_{p, \beta+p-1}$ with $p \in \{2, \dots, \alpha+1\}$ lie above the staircase, and thus $\Alm_{1,\beta}$ must lie below the staircase, so as not to contradict the maximality of $\alpha$. Thus, we have $\Alm_{k,\beta+1}$ above the staircase and $\Alm_{1,\beta}$ below the staircase and so the staircase must be:  

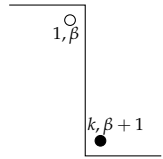
\begin{center}
\begin{tikzpicture}[scale=0.5]
 \draw (6.5,4.5) -- (8.5,4.5)--(8.5,0.5)--(10.5,0.5);
\node[vertex] at (8.1,4.1) {};
\node at (8,3.7) {\scalebox{0.6}{$1,\beta$}}; 
\node[fvertex] at (8.9,0.9) {};
\node at (9.3,1.3) {\scalebox{0.6}{$k,\beta+1$}}; 
\end{tikzpicture}
\end{center}

If $\Alm_{\alpha,1}$ lies above the staircase then all the points $\Alm_{p, \beta+p-1}$ with $p \in \{1, \dots, \alpha\}$ must lie above the staircase, and thus $\Alm_{\alpha+1,k}$ must lie below the staircase, so as not to contradict the maximality of $\alpha$. Hence, we have $\Alm_{\alpha,1}$ above the staircase and $\Alm_{\alpha+1,k}$ below the staircase and so the staircase must be:  
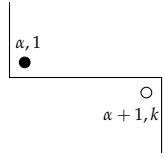
\begin{center}
\begin{tikzpicture}[scale=0.5]
 \draw (12.5,4.5) -- (12.5,2.5)--(16.5,2.5)--(16.5,0.5);
\node[vertex] at (16.1,2.1) {};
\node at (13,3.4) {\scalebox{0.6}{$\alpha,1$}}; 
\node[fvertex] at (12.9,2.9) {};
\node at (15.7,1.5) {\scalebox{0.6}{$\alpha+1,k$}};
\end{tikzpicture}
\end{center}

\end{proof}

\subsection{Dimension Formula} \label{S:dimension}
Now we may use the combinatorial tool developed Section \ref{S:tool} to provide a formula for calculating the dimension of $\Ext^1(I(\unl), I(\um))$.
\begin{theorem}\label{T: ext dimension}
Given two $k$-subsets $\unl$ and $\um$, 
\begin{align*}
\cdim(\Ext^1(I(\unl), I(\um)))=  \alpha(\unl,\um) +\beta(\unl,\um) - k - | \unl \cap \um |. 
\end{align*}
\end{theorem}

\begin{remark}
By Proposition \ref{P: stably 2CY}, the subcategory $\mathrm{MCM_{\bZ}^0}R$ of $\mathcal{C}$ consisting of generically free modules is stably 2-Calabi-Yau, and thus for any two generically free modules $M$ and $N$ in $\mathcal{C}$, it immediately follows that 
\begin{align*}
\cdim(\Ext^1(M, N))=\cdim(\Ext^1(N, M)).
\end{align*}
It is easy to see that by combining Lemma \ref{L:symmetry} and Theorem \ref{T: ext dimension}, our combinatorial tool allows us to verify this symmetry directly for the generically free modules of rank 1.
\end{remark}

To prove Theorem \ref{T: ext dimension}, we start by fixing the following notation:
\begin{itemize}
\item 
$I \colonequals I(\unl) =(x^{k-1}, x^{k-2}y^{i_1}, x^{k-3}y^{i_2}, \ldots, xy^{i_{k-2}}, y^{i_{k-1}})(i_k) $
where $i_k = k-1-\ell_k$ and $i_{k-p}=\ell_k-\ell_p-(k-p)$;
\item $J \colonequals I(\um) =(x^{k-1}, x^{k-2}y^{j_1}, x^{k-3}y^{j_2}, \ldots, xy^{j_{k-2}}, y^{j_{k-1}})(j_k) $
where $j_k = k-1-m_k$ and $j_{k-p}=m_k-m_p-(k-p)$;
\item $\mathbf{J} \colonequals J(\deg_I(x^{k-1})) \oplus J(\deg_I(x^{k-2}y^{i_1})) \oplus \dots \oplus J(\deg_I(y^{i_{k-1}}))$. This means that an element of $\mathbf{J}$ is a vector with $m$-th component in the ideal $J$ shifted by $\deg_I(x^{k-m}y^{i_{m-1}})$.
\end{itemize}

Our first observation is that we may assume that $i_k=0$, or equivalently $\ell_k=k-1$. Indeed, if this doesn't hold, we may shift both $I$ and $J$ by $-i_k$ to get to this setting, which will not affect the Ext calculation as we have only shifted the grading. Moreover, this corresponds to shifting both $\unl$ and $\um$ by $i_k$ and so it does not change $\Alm$ or $\Blm$ in any way. For future use, also note that 
\begin{align}
m_p = -j_k -j_{k-p}+p-1 \quad \text{and} \quad \ell_p = -i_{k-p}+p-1 \label{pluckercoord}
\end{align}
where, for ease of notation, we set $i_0=0=j_0$.

\subsubsection{Matrix Factorisations}
In the ring $R= \mathbb{C}[x,y]/(x^k)$ a matrix factorisation for the ideal 
\[
I = (x^{k-1},x^{k-2}y^{i_1},x^{k-3}y^{i_2}, \ldots, xy^{i_{k-2}}, y^{i_{k-1}})
\]
where $0 \leq i_1 \leq i_2 \leq \ldots \leq i_{k-1}$ is given as 
\[
R^k \xrightarrow{M} R^k \xrightarrow{N} R^k \to I \to 0
\]
where $M, N$ are the $k \times k$ upper triangular matrices: 

\begin{align*}
 M =
\begin{pmatrix}
x^{k-1} & x^{k-2}y^{i_1} & x^{k-3}y^{i_2} & \ldots & xy^{i_{k-2}} & y^{i_{k-1}}\\
0 & x^{k-1} & x^{k-2}y^{i_2-i_1} & \ldots & x^2y^{i_{k-2}-i_1} & xy^{i_{k-1}-i_1}\\
0 & 0 & x^{k-1} & \dots & x^3y^{i_{k-2}-i_2} & x^2y^{i_{k-1}-i_2}\\
& & & \ddots & \vdots & \vdots \\
& & & & x^{k-1} & x^{k-2}y^{i_{k-1}-i_{k-2}}\\
& & & & 0 & x^{k-1}\\
\end{pmatrix}
\end{align*}

and 

\begin{align*}
N = \begin{pmatrix}
x & -y^{i_1} & 0 & 0 & &   \\
0 & x & -y^{i_2-i_1} & 0 & &  \\
0 & 0 & x & -y^{i_3-i_2} & &   \\
&&& \ddots & \ddots & & \\
&&&& x & -y^{i_{k-2}-i_{k-3}} & 0\\
&&&& 0 & x & -y^{i_{k-1}-i_{k-2}}  \\
&&& & 0 & 0 & x  \\
\end{pmatrix}.
\end{align*}

In particular, a graded projective presentation of $I$ is 
\begin{align*}
R(-\deg_I(x^{k-1})-k) \oplus R(-\deg_I(&x^{k-2}y^{i_1})-k) \oplus \dots \oplus R(-\deg_I(y^{i_{k-1}})-k) \\
&\downarrow M\\
 R(-\deg_I(x^{k-1})-1) \oplus R(-\deg_I(&x^{k-2}y^{i_1})-1) \oplus \dots \oplus R(-\deg_I(y^{i_{k-1}})-1)\\ 
&\downarrow N\\
R(-\deg_I(x^{k-1})) \oplus R(-\deg_I(&x^{k-2}y^{i_1})) \oplus \dots \oplus R(-\deg_I(y^{i_{k-1}})) \\
&\downarrow \\
 &I \\
&\downarrow\\
 &0
\end{align*}

We remark that the matrix factorizations are not  reduced if some of the $i_j$'s are equal.

\subsubsection{Strategy}
\label{S: Strategy} To calculate $\Ext^1(I,J)$, take the graded projective presentation of $I$ above and apply the graded $\Hom^\mathbb{Z}(-,J)= \bigoplus_{n\in \mathbb{Z}} \Hom(-,J(n))$. Since $\Hom^\mathbb{Z}(R(a),J) \cong J(-a)$, this gives
\begin{align*}
\mathbf{J} \xrightarrow{N^T} \mathbf{J}(1) \xrightarrow{M^T} \mathbf{J}(k) 
\end{align*}
and to obtain $\Ext^1(I,J)$, we calculate $(\ker(M^T)/\mathrm{im}(N^T))_0$ or equivalently $\ker(M^T)_0/\mathrm{im}(N^T)_0$. In fact, we will only be interested in the dimension of the Ext group which we can calculate as 
\begin{align*}
\cdim(\Ext^1(I,J))=\cdim(\ker(M^T)_0) - \cdim(\mathrm{im}(N^T)_0).
\end{align*}
Since the maps are graded and each of the degree zero parts are finite dimensional $\mathbb{C}$-vector spaces, we may use the standard rank-nullity theorem to say
\begin{align*}
\cdim(\ker(M^T)_0)= \cdim( \mathbf{J}(1)_0) - \cdim(\mathrm{im}(M^T)_0)
\end{align*}
and 
\begin{align*}
\cdim(\mathrm{im}(N^T)_0)= \cdim(\mathbf{J}_0) - \cdim(\ker(N^T)_0).
\end{align*}
So our strategy to prove Theorem \ref{T: ext dimension} is to determine the complex dimensions of $\mathbf{J}_0, \mathbf{J}(1)_0, \ker(N^T)_0$ and $\mathrm{im}(M^T)_0$, and then to combine them to determine $\cdim(\Ext^1(I,J))$.
\subsubsection{Calculating Dimensions}
\begin{lemma} \label{lem: generic first term}
A degree zero element of $\mathbf{J}$ has the following form,
\begin{align*}
\underline{a} = 
\begin{pmatrix}
a_{11}x^{k-1}y^{-j_{k}}&+ &a_{12}x^{k-2}y^{-j_{k}-1} &+& \dots& + &a_{1k} y^{-j_{k}+1-k} \\
a_{21}x^{k-1}y^{i_{1}-j_{k}+1}&+ &a_{22}x^{k-2}y^{i_{1}-j_{k}} &+& \dots &+& a_{2k} y^{i_{1}-j_{k}+2-k} \\
\vdots & \vdots & \vdots & \vdots &  & \vdots & \vdots\\
a_{k1}x^{k-1}y^{i_{k-1}-j_{k}+k-1}&+& a_{k2}x^{k-2}y^{i_{k-1}-j_{k}+k-2}& +& \dots& + &a_{kk} y^{i_{k-1}-j_{k}}
\end{pmatrix}
\end{align*} 
where $a_{pq}$ is the coefficient of $x^{k-q}y^{p-q+i_{p-1}-j_k}$ and $a_{pq} \in \mathbb{C}$ can be nonzero if and only if $\ell_{k+1-p} \leq m_{k+1-q}$.
\end{lemma}
\begin{proof}
Recall that $J(n)_0=J_n$ and thus, remembering that $J$ is an ideal shifted by $j_k$, monomials $x^ay^b$ lie in $J(n)_0$ precisely when $a-b=n+j_k$. In particular, if $a=k-q$ for some $q=1, \dots, k$, then 
\begin{align*}
b=k-q-n-j_k.
\end{align*}
Therefore, there is a $\mathbb{C}$-basis for $J(n)_0$ which consists of the subset of
\begin{align*}
x^{k-1}y^{k-1-n-j_{k}}, x^{k-2}y^{k-2-n-j_{k}}, \dots, xy^{1-n-j_{k}}, y^{-n-j_{k}}
\end{align*}
which lie in $J$. In particular, for each $p=1, \dots, k$ a degree zero element of $J(\deg(x^{k-p}y^{i_{p-1}}))$ is
\begin{align*}
\sum_{q=1}^k a_{pq}x^{k-q}y^{p-q+i_{p-1}-j_k}
\end{align*}
where $a_{pq}$ can be nonzero only if $x^{k-q}y^{p-q+i_{p-1}-j_k} \in J$, or equivalently,
\begin{align*}
p-q+i_{p-1}-j_k \geq j_{q-1} &\iff -j_{q-1}-j_k -q \geq -i_{p-1}-p \\
&\iff -j_{q-1}-j_k + (k+1-q)- 1 \geq-i_{p-1}+ (k +1-p)-1\\
&\iff m_{k+1-q} \geq \ell_{k+1-p}.  \tag{using \eqref{pluckercoord}}
\end{align*} \qedhere
\end{proof}

\begin{example} \label{Ex:bfJandstaircase}
Take $k=3$ and consider $\unl=(-2,0,2)$ and $\um=(-1,2,3)$. These correspond to ideals
\begin{align*}
I= (x^2,xy,y^2) \quad \text{and} \quad J=(x^2,x,y^2)(-1) .
\end{align*}
In this case a degree zero element of $\mathbf{J}=J(2) \oplus J(0) \oplus J(-2)$ is
\begin{align} \label{eq: ex matrix of coeff}
\begin{pmatrix}
a_{11}x^2y &+ &a_{12}x & & \\
a_{21}x^2y^3 &+ & a_{22}xy^2 & &\\
a_{31}x^2y^5 &+ & a_{32}xy^4 &+ & a_{33}y^3
\end{pmatrix}
\end{align}
where $a_{ij} \in \mathbb{C}$. Notice that $a_{13}$ and $a_{23}$ do not appear since $y^{-1}, y \notin J$. Compare this with $\Alm$, and its image after rotating by a half turn:
\begin{center}
  \begin{tikzpicture}[scale=0.67]
\node at (-1.5,2) {$\Alm=$};
    \foreach \x in {1,2,3}
    \foreach \y in {1,2,3}
    {
    \node[vertex] at (\x,\y) {};
    }
    \fill (1,4-1) node[fvertex] {};
    \fill (2,4-1) node[fvertex] {};
    \fill (3,4-1) node[fvertex] {};
    \fill (2,4-2) node[fvertex] {};
    \fill (3,4-2) node[fvertex] {};
    \fill (2,4-3) node[fvertex] {};
    \fill (3,4-3) node[fvertex] {};
 \draw (0.5,3.5) -- (0.5,2.5)--(1.5,2.5)--(1.5,0.5)--(3.5,0.5);
  \end{tikzpicture}
\hspace{2cm}
  \begin{tikzpicture}[scale=0.67]
\node at (-2.5,2) {\text{`rotated'} $\Alm=$};
    \foreach \x in {1,2,3}
    \foreach \y in {1,2,3}
    {
    \node[vertex] at (\x,\y) {};
    }
    \fill (1,4-1) node[fvertex] {};
    \fill (2,4-1) node[fvertex] {};
    \fill (1,4-2) node[fvertex] {};
    \fill (2,4-2) node[fvertex] {};
    \fill (1,4-3) node[fvertex] {};
    \fill (2,4-3) node[fvertex] {};
    \fill (3,4-3) node[fvertex] {};
 \draw (0.5,3.5) -- (2.5,3.5)--(2.5,1.5)--(3.5,1.5)--(3.5,0.5);
  \end{tikzpicture}
\end{center}
After rotation, the shape formed by the staircase path is precisely the same as that of the possibly nonzero coefficients in \eqref{eq: ex matrix of coeff}. This follows since $a_{ij}$ can be nonzero if and only if $\ell_{k+1-i} \leq m_{k+1-j}$ which by definition is if and only if $\Alm_{k+1-i,k+1-j}$ is shaded.
\end{example}

\begin{lemma} \label{lem: dimension first term}
With the setup above, $\cdim(\mathbf{J}_0) = | \{ (i,j) \mid \text{$1 \leq i,j \leq k$ and $\Alm_{i,j}$ is shaded}\}|$.
\end{lemma}
\begin{proof}
It is clear from Lemma \ref{lem: generic first term} that 
\begin{align*}
\cdim(\mathbf{J}_0) = | \{ (i,j) \mid \text{$1 \leq i,j \leq k$ and $a_{ij}$ can be nonzero}\}|.
\end{align*}
Moreover, we know that $a_{ij}$ can be nonzero precisely when $\ell_{k+1-i} \leq m_{k+1-j}$ which, by definition, is if and only if $\Alm_{k+1-i,k+1-j}$ is shaded. Thus, 
\begin{align*}
\cdim(\mathbf{J}_0) = | \{ (i,j) \mid \text{$1 \leq i,j \leq k$ and $\Alm_{k+1-i,k+1-j}$ is shaded}\}|.
\end{align*}
The map $(i,j) \mapsto (k+1-i,k+1-j)$ precisely describes the rotation of the $(k \times k)$-grid as seen in Example \ref{Ex:bfJandstaircase}. Since this gives a bijection from $\{1, \dots, k\} \times \{1, \dots, k\}$ to itself, the right hand side is the same as $| \{ (i,j) \mid \text{$1 \leq i,j \leq k$ and $\Alm_{i,j}$ is shaded}\}|$ completing the proof.
\end{proof}

\begin{lemma} \label{lem: generic second term}
A degree zero element of $\mathbf{J}(1)$ has the following form,
\begin{align*}
\underline{b} = 
\begin{pmatrix}
b_{11}x^{k-1}y^{-j_{k}-1}&+ &b_{12}x^{k-2}y^{-j_{k}-2} &+& \dots& + &b_{1k} y^{j_{k}-k} \\
b_{21}x^{k-1}y^{i_{1}-j_{k}}&+ &b_{22}x^{k-2}y^{i_{1}-j_{k}-1} &+& \dots &+& b_{2k} y^{i_{1}-j_{k}+1-k} \\
\vdots & \vdots & \vdots & \vdots &  & \vdots & \vdots\\
b_{k1}x^{k-1}y^{i_{k-1}-j_{k}+k-2}&+& b_{k2}x^{k-2}y^{i_{k-1}-j_{k}+k-3}& +& \dots& + &b_{kk} y^{i_{k-1}-j_{k}-1}
\end{pmatrix}
\end{align*} 
where $b_{pq}$ is the coefficient of $x^{k-q}y^{p-q+i_{p-1}-j_k-1}$ and $b_{pq} \in \mathbb{C}$ can be nonzero if and only if $\ell_{k+1-p} < m_{k+1-q}$.
\end{lemma}
\begin{proof}
This proof is completely analogous to the proof of Lemma \ref{lem: generic first term}. For each $p=1, \dots, k$ a degree zero element of $J(\deg(x^{k-p}y^{i_{p-1}})+1)$ is
\begin{align*}
\sum_{q=1}^k b_{pq}x^{k-q}y^{p-q+i_{p-1}-j_k-1}
\end{align*}
(the $y$-index drops by one from Lemma \ref{lem: generic first term} since we have shifted by one) where $b_{pq}$ can be nonzero only if $x^{k-q}y^{p-q+i_{p-1}-j_k-1} \in J$, or equivalently,
\begin{align*}
p-q+i_{p-1}-j_k-1 \geq j_{q-1} &\iff p-q+i_{p-1}-j_k > j_{q-1} \\
&\iff -j_{q-1}-j_k -q > -i_{p-1}-p \\
&\iff -j_{q-1}-j_k + (k+1-q)- 1 >-i_{p-1}+ (k +1-p)-1\\
&\iff m_{k+1-q} > \ell_{k+1-p}.  \tag{using \eqref{pluckercoord}}
\end{align*} \qedhere 
\end{proof}
\begin{lemma} \label{lem: dimension second term}
With the setup above, $\cdim(\mathbf{J}(1)_0) = | \{ (i,j) \mid \text{$1 \leq i,j \leq k$ and $\Blm_{i,j}$ is shaded}\}|$.
\end{lemma}
\begin{proof}
Completely analogous to Lemma \ref{lem: dimension first term}, but now using that, by definition, $m_{k+1-q} > \ell_{k+1-p}$ if and only if $\Blm_{k+1-i,k+1-j}$ is shaded.
\end{proof}
\begin{corollary}\label{cor: dimension difference}
With the setup above, $\cdim(\mathbf{J}_0)-\cdim(\mathbf{J}(1)_0)= | \unl \cap \um|$.
\end{corollary} 
\begin{proof}
Using Lemmas \ref{lem: dimension first term} and \ref{lem: dimension second term},
\begin{align*}
\cdim(\mathbf{J}_0)-\cdim(\mathbf{J}(1)_0) = | \{ (i,j) &\mid \text{$1 \leq i,j \leq k$ and $\Alm_{i,j}$ is shaded}\}| \\
&- | \{ (i,j) \mid \text{$1 \leq i,j \leq k$ and $\Blm_{i,j}$ is shaded}\}|.
\end{align*}
Since $\ell_i < m_j$ implies $\ell_i \leq m_j$, it is clear that if $\Blm_{i,j}$ is shaded then so is $\Alm_{i,j}$ and hence the right hand side is simply 
\begin{align*}
| \{ (i,j) \mid \text{$1 \leq i,j \leq k$ and $\Alm_{i,j}$ is shaded and $\Blm_{i,j}$ is empty}\}| \\
=| \{ (i,j) \mid \text{$1 \leq i,j \leq k$ and $\ell_i \leq m_j$ and $\ell_i \geq m_j$}\}| \\
= | \{ (i,j) \mid \text{$1 \leq i,j \leq k$ and $\ell_i = m_j$}\}|.
\end{align*}
For each such pair $(i,j)$, it is clear there is a corresponding element of $\unl \cap \um$ and since $\unl$ and $\um$ are strictly increasing sequences, each element of $\unl \cap \um$ corresponds to a unique such pair $(i,j)$. Thus,
\begin{align*}
 | \{ (i,j) \mid \text{$1 \leq i,j \leq k$ and $\ell_i = m_j$}\}|=| \unl \cap \um |
\end{align*}
and so $\cdim(\mathbf{J}_0)-\cdim(\mathbf{J}(1)_0)=| \unl \cap \um |$ as required.
\end{proof}

\begin{lemma} \label{lem: kernel dimension}
With the setup above, $\cdim(\ker(N^T)_0)= \alpha(\unl,\um)$.
\end{lemma}
A calculation for $\cdim(\ker(N^T)_0)$ when $k=3$ is given in Example \ref{ex:dim ker N}.
\begin{proof}
By Lemma \ref{lem: generic first term}, we know the form of a generic element $\underline{a}$ of $\mathbf{J}$ and applying $N^T$ gives a vector $N^T(\underline{a})$ with first term
\begin{align*}
 \sum_{q=1}^{k-1} a_{1q+1} x^{k-q}y^{-j_k -q}
\end{align*}
and for $2 \leq p \leq k$, its $p$-th term is 
\begin{align*}
\left( \sum_{q=1}^{k-1} (a_{pq+1}-a_{p-1q}) x^{k-q}y^{-j_k +i_{p-1}+p-q} \right) - a_{p-1k}y^{-j_k +i_{p-1}+p-q}.
\end{align*}
In particular, $\underline{a} \in \ker(N^T)$ if and only if the coefficient of each monomial in each of these expressions is zero i.e.\
\begin{enumerate}
\item $a_{1q}=0$ for all $q=2, \dots, k$;
\item $a_{pk}=0$ for all $p=1, \dots, k-1$;
\item $a_{pq+1}=a_{p-1q}$ for all $p=2, \dots, k$, $q=1, \dots, k-1$ or equivalently, $a_{p+1q+1}=a_{pq}$ for all $1 \leq p,q \leq k-1$.
\end{enumerate}
Note that $(3)$ holds if and only if, in the matrix of coefficients
\begin{equation}\label{eq: coefficient matrix 1}
\begin{array}{c}
\begin{tikzpicture}
\node at (0,0) {$\begin{pmatrix}
a_{11} & a_{12} & \dots & a_{1k-1} & a_{1k}\\
a_{21} & a_{22} & \dots & a_{2k-1} & a_{2k}\\
\vdots & \vdots & \ddots  & \vdots & \vdots \\
a_{k-11} & a_{k-12} & \dots & a_{k-1k-1} & a_{k-1k}\\
a_{k1} & a_{k2} & \dots & a_{kk-1} & a_{kk}
\end{pmatrix}$}; 
\draw[gray] (-2.3,1) -- (2.3,-1);
\draw[gray] (-1.1,1) -- (2.3,-0.5);
\draw[gray] (1,1) -- (2.3,0.5);
\draw[gray] (0,1) -- (2.3,0);
\draw[gray] (2,1.1) -- (2.4,0.85);
\draw[gray] (-2.3,0.5) -- (1,-1);
\draw[gray] (-2.3,0) -- (-0.2,-1);
\draw[gray] (-2.3,-0.5) -- (-1.1,-1);
\draw[gray] (-2.4,-0.9) -- (-2.2,-1.1);
\end{tikzpicture}
\end{array}
\end{equation}
the value along each of the diagonals is constant. If we further impose conditions $(1)$ and $(2)$, this shows that each of the diagonals above the main diagonal must be zero. Thus, $\underline{a} \in \ker(N^T)$ if and only if the $a_{ij}$ above the main diagonal are zero, and the $a_{ij}$ on each lower diagonal are constant. Hence, we see that for an element of $\ker(N^T)_0$, there are at most $k$-free choices - one for each of the lower diagonals. However, for a given diagonal $D_p^-$, (where here we are abusing notation by using $D_p^\pm$ to denote diagonals in matrices, as well as in $(k \times k)$ grids) we may only choose something nonzero if \underline{all} the $a_{ij}$ along $D_p^-$ may be nonzero. Thus $\cdim(\ker(N^T)_0)$ is precisely the number of lower diagonals in the matrix \eqref{eq: coefficient matrix 1} along which all $a_{ij}$ may be nonzero.

Recall from Lemma \ref{lem: generic first term} that $a_{ij}$ can be nonzero if and only if $\ell_{k+1-i} \leq m_{k+1-j}$ if and only if the entry $\Alm_{k+1-i, k+1-j}$ is shaded. Thus a lower diagonal  in the matrix \eqref{eq: coefficient matrix 1}, say $D_p^-$, can be all nonzero if and only if the upper diagonal $D_p^+$ in $\Alm$ is completely shaded. Hence, the number of lower diagonals in \eqref{eq: coefficient matrix 1} along which all entries may be nonzero is the same as the number of upper diagonals in $\Alm$ which are completely above the staircase path which, by definition, is $\alpha(\unl,\um)$.
\end{proof}
\begin{lemma} \label{cor: image dimension}
With the setup above, $\cdim(\mathrm{im}(M^T)_0)= k -\beta(\unl,\um)$.
\end{lemma}
Another computation for $\cdim(\mathrm{im}(M^T)_0)$ in the $k=3$ case will be shown in Example \ref{ex: dim im M}.
\begin{proof}
Recall that $(M^T)_{p,q}=0$ if $p<q$ and $(M^T)_{p,q}=x^{k-1-p+q}y^{i_{p-1}-i_{q-1}}$ if $p \geq q$. Applying $M^T$ to a generic element $\underline{b}$ of $\mathbf{J}(1)_0$ (cf.~Lemma \ref{lem: generic second term}) is therefore 
\begin{align*}
(M^T(\underline{b}))_p &= \sum_{q=1}^k (M^T)_{p,q} (\underline{b})_q\\
		&= \sum_{q=1}^k (M^T)_{p,q} \left( \sum_{r=1}^k b_{qr}x^{k-r}y^{i_{q-1}-j_k-1+q-r} \right)\\
		&= \sum_{q=1}^p x^{k-1-p+q}y^{i_{p-1}-i_{q-1}}  \left( \sum_{r=1}^k b_{qr}x^{k-r}y^{i_{q-1}-j_k-1+q-r} \right)\\
		&=\sum_{q=1}^p \sum_{r=1}^k b_{qr} x^{2k-1-p+q-r}y^{i_{p-1}-j_k-1+q-r}. 
\end{align*} 
Recall that $x^k=0$ and so for a term $x^{2k-1-p+q-r}$ to be nonzero, it must be that 
\begin{align*}
2k-1-p+q-r \leq k-1 \iff r \geq k-p+q.
\end{align*}
Thus we may write 
\begin{align}
(M^T(\underline{b}))_p =\sum_{q=1}^p \sum_{r=k-p+q}^k b_{qr} x^{2k-1-p+q-r}y^{i_{p-1}-j_k-1+q-r}. \label{big sum}
\end{align}
Now set $s=k+q-r$. Then, since $r \leq k$ we have $s =k+q-r \geq k+q-k =q$.
Moreover, since $r \geq k-p+q$ we have $s =k+q-r \leq k+q-(k-p+q) =p$. Thus we may reindex \eqref{big sum} to get
\begin{align*}
(M^T(\underline{b}))_p &=\sum_{q=1}^p \sum_{s=q}^p b_{qk+q-s} x^{k-1-p+s}y^{i_{p-1}-j_k-1-k+s}\\
&=\sum_{s=1}^p \sum_{q=1}^s b_{qk+q-s} x^{k-1-p+s}y^{i_{p-1}-j_k-1-k+s}\\
&= \sum_{s=1}^px^{k-1-p+s}y^{i_{p-1}-j_k-1-k+s} \left(\sum_{q=1}^s  b_{qk+q-s}  \right).
\end{align*}

Notice that for each $1 \leq s \leq k$ the complex number $\upgamma_s \colonequals \sum_{q=1}^s b_{q,k+q-s}$ appears as a coefficient in the terms $(M^T(\underline{b}))_p$ for $p=s, \dots, k$, and the $\gamma_s$ are mutually independent, as none of the $b_{ij}$ appear as a summand in more than one $\gamma_s$. In particular, we may write $M^T(\underline{b})$ as
\begin{align*}
\upgamma_1
\begin{pmatrix}
x^{k-1}y^{-j_k-k}\\
x^{k-2}y^{i_1-j_k-k}\\
\vdots\\
xy^{i_{k-2}-j_k-k}\\
y^{i_{k-1}-j_k-k}
\end{pmatrix}
+\upgamma_2
\begin{pmatrix}
0\\
x^{k-1}y^{i_1-j_k-(k-1)}\\
\vdots\\
x^2y^{i_{k-2}-j_k-(k-1)}\\
xy^{i_{k-1}-j_k-(k-1)}
\end{pmatrix}+ \dots 
+\upgamma_k
\begin{pmatrix}
0\\
0\\
\vdots\\
0\\
x^{k-1}y^{i_{k-1}-j_k-1}
\end{pmatrix}.
\end{align*}
and thus the dimension of $\mathrm{im}(M^T)_0$ is the number of these vectors whose corresponding coefficient $\upgamma_s$ may be nonzero. But $\upgamma_s$ may be nonzero if and only if at least one of the $b_{qk+q-s}$ for $q=1, \dots, s$ may be nonzero and so the dimension of $\mathrm{im}(M^T)_0$ is the number of upper diagonals in the coefficient matrix
\begin{equation}
\begin{array}{c} \label{eq: coefficient matrix 2}
\begin{tikzpicture}
\node at (0,0) {$\begin{pmatrix}
b_{11} & b_{12} & \dots & b_{1k-1} & b_{1k}\\
b_{21} & b_{22} & \dots & b_{2k-1} & b_{2k}\\
\vdots & \vdots & \ddots  & \vdots & \vdots \\
b_{k-11} & b_{k-12} & \dots & b_{k-1k-1} & b_{k-1k}\\
b_{k1} & b_{k2} & \dots & b_{kk-1} & b_{kk}
\end{pmatrix}$}; 
\draw[gray] (-2.3,1) -- (2.3,-1);
\draw[gray] (-1.1,1) -- (2.3,-0.5);
\draw[gray] (1,1) -- (2.3,0.5);
\draw[gray] (0,1) -- (2.3,0);
\draw[gray] (2,1.1) -- (2.4,0.85);
\end{tikzpicture}
\end{array}
\end{equation}
where at least one coefficient along that diagonal can be nonzero. Equivalently, the dimension $\cdim(\mathrm{im}(M^T)_0)$ is precisely $k$ minus the number of upper diagonals where all the coefficients must be zero.\\
Recall that $b_{ij}$ can be nonzero if and only if $\ell_{k+1-i} < m_{k+1-j}$ if and only if $\Blm_{k+1-i, k+1-j}$ is shaded. Thus the upper diagonal  $D_p^+$ in the matrix \eqref{eq: coefficient matrix 2} 
has to be all zero if and only if the lower diagonal $D_p^-$ in $\Blm_{i,j}$ 
is completely unshaded. Hence, the number of upper diagonals in \eqref{eq: coefficient matrix 2} along which all entries have to be zero is the same as the number of lower diagonals in $\Blm$ which are completely below the staircase path which, by definition, is $\beta(\unl,\um)$.
\end{proof}

We are now ready to prove Theorem \ref{T: ext dimension}.
\begin{proof}[Proof of Theorem \ref{T: ext dimension}]
As explained in Section \ref{S: Strategy}, $\cdim(\Ext^1(I(\unl), I(\um)))$ is calculated as
\begin{align*}
\cdim(\ker(M^T)_0) - \cdim(\mathrm{im}(N^T)_0)
\end{align*}
which by rank-nullity is equal to
\begin{align*}
(\cdim(\mathbf{J}(1)_0)- \cdim(\mathrm(M^T)_0)) - (\cdim(\mathbf{J}_0)-\cdim(\ker(N^T)_0)). 
\end{align*}
Lemma \ref{cor: dimension difference} shows
\begin{align*}
\cdim(\mathbf{J}(1)_0)-\cdim(\mathbf{J}_0)= - |\unl \cap \um |
\end{align*}
and Lemmas \ref{lem: kernel dimension} and \ref{cor: image dimension} respectively show that
\begin{align*}
\cdim(\ker(N^T)_0)= \alpha(\unl,\um) \quad \text{and} \quad \cdim(\mathrm{im}(M^T)_0)= k -\beta(\unl,\um)
\end{align*}
so combining all of these gives
\begin{align*}
\cdim(\Ext^1(I(\unl), I(\um)))= \beta(\unl,\um) + \alpha(\unl,\um) - k - | \unl \cap \um |. 
\end{align*}
\end{proof}
This gives all we need to prove our main result in the special case when $\unl$ and $\um$ are disjoint.
\begin{corollary} \label{cor: disjoint noncrossing}
Given two disjoint $k$-subsets $\unl$ and $\um$, $\cdim(\Ext^1(I(\unl), I(\um)))=0$ if and only if $\unl$ and $\um$ are noncrossing.
\end{corollary}
\begin{proof}
Since $\unl$ and $\um$ are disjoint, $|\unl \cap \um |=0$ and so by Theorem \ref{T: ext dimension}
\begin{align*}
\cdim(\Ext^1(I(\unl), I(\um)))= \beta(\unl,\um) + \alpha(\unl,\um) - k. 
\end{align*}
Thus, $\Ext^1(I(\unl), I(\um))=0$ if and only if $\beta(\unl,\um) + \alpha(\unl,\um) = k$ which holds if and only if $\unl$ and $\um$ are noncrossing by Lemma \ref{lem: single step}. 
\end{proof}

\subsection{The \texorpdfstring{$k=3$}{k=3} Case} \label{S:example}

In this subsection, we will illustrate the calculations in Lemma \ref{lem: kernel dimension}, and Lemma \ref{cor: image dimension} in the $k=3$ case.

The following example will demonstrate Lemma \ref{lem: kernel dimension}, showing that $\cdim(\ker(N^T)_0)= \alpha(\unl,\um)$.

\begin{example}[$k=3$ example] \label{ex:dim ker N}
Applying $N^T$ to a generic element $\underline{a}$ of $\mathbf{J}=J(2) \oplus J(1-i_1) \oplus J(-i_2)$ gives
\begin{align*}
N^T(\underline{a})= 
\begin{pmatrix}
a_{12}x^2y^{-j_3-1} + a_{13}xy^{-j_{3}-2}\\
(a_{22}-a_{11})x^2y^{i_1-j_3} + (a_{23}-a_{12})xy^{i_1-j_{3}-1}- a_{13}y^{i_1-j_{3}-2}\\
(a_{32}-a_{21})x^2y^{i_2-j_3+1} + (a_{33}-a_{22})xy^{i_2-j_{3}}- a_{23}y^{i_2-j_{3}-1}
\end{pmatrix}.
\end{align*}
In particular, $\underline{a}$ lies in $\ker(N^T)_0$ if and only if 
\begin{align*}
a_{11}=a_{22}=a_{33}, \ a_{32}=a_{21}, \ \text{and} \ a_{12}=a_{13}=a_{23}=0.
\end{align*}
Equivalently, in the matrix of coefficients
\begin{center}
\begin{tikzpicture}
\node at (0,0) {$\begin{pmatrix}
a_{11} & a_{12} & a_{13}\\
a_{21} & a_{22} & a_{23}\\
a_{31} & a_{32} & a_{33}
\end{pmatrix}$}; 
\draw[gray] (-1,0.5) -- (1,-0.5);
\draw[gray] (-1,0) -- (0,-0.5);
\draw[gray] (-1.1,-0.45) -- (-0.9,-0.55);
\end{tikzpicture}
\end{center}
all the entries above the main diagonal must be zero, and those connected by a line must all be equal. Thus, the dimension of  $\ker(N^T)_0$ is at most three, with possible basis vectors corresponding to each of these lines:
\begin{align*}
\begin{pmatrix} 
x^2y^{-j_3}\\
xy^{i_1-j_3}\\
y^{i_2-j_3}
\end{pmatrix} \in \ker(N^T)_0 &\iff a_{11}, a_{22}, a_{33} \ \text{can all be nonzero}\\
&\iff  \ell_1 \leq m_1, \ell_2 \leq  m_2, \ell_3 \leq m_3
\end{align*}
which by definition is if and only if, in $\Alm$, all those vertices in diagonal $D_3^+$, circled below,
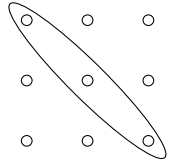
\begin{center}
\begin{tikzpicture}[scale=0.8]
    \foreach \x in {1,2,3}
    \foreach \y in {1,2,3}
    {
    \node[vertex] at (\x,\y) {};
    }
\draw[rotate=45] (2.825,0) ellipse (0.3cm and 1.8cm);
  \end{tikzpicture}
\end{center}
are shaded, or equivalently, this diagonal lies completely above the corresponding staircase path. Similarly, 
\begin{align*}
\begin{pmatrix} 
0\\
x^2y^{i_1-j_3+1}\\
xy^{i_2-j_3+1}
\end{pmatrix} \in \ker(N^T)_0 &\iff a_{21}, a_{32} \ \text{can both be nonzero}\\
&\iff  \ell_2 \leq m_3, \ell_1 \leq  m_2
\end{align*}
which by definition is if and only if, in $\Alm$, all those vertices in diagonal $D_2^+$, circled below, 
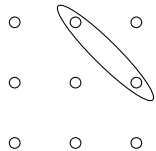
\begin{center}
\begin{tikzpicture}[scale=0.8]
    \foreach \x in {1,2,3}
    \foreach \y in {1,2,3}
    {
    \node[vertex] at (\x,\y) {};
    }
\draw[rotate=45] (3.525,0) ellipse (0.2cm and 1.1cm);
  \end{tikzpicture}
\end{center}
are shaded, or equivalently, this diagonal lies completely above the corresponding staircase path.
And finally, 
\begin{align*}
\begin{pmatrix} 
0\\
0\\
x^2y^{i_2-j_3+2}
\end{pmatrix} \in \ker(N^T)_0 &\iff a_{31} \ \text{can be nonzero}\\
&\iff  \ell_1 \leq m_3
\end{align*}
which by definition is if and only if, in $\Alm$, all those vertices in the circled diagonal $D_1^+$
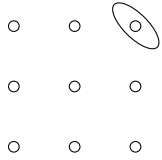
\begin{center}
\begin{tikzpicture}[scale=0.8]
    \foreach \x in {1,2,3}
    \foreach \y in {1,2,3}
    {
    \node[vertex] at (\x,\y) {};
    }
\draw[rotate=45] (4.25,0) ellipse (0.2cm and 0.5cm);
  \end{tikzpicture}
\end{center}
are shaded, or equivalently, this diagonal lies completely above the corresponding staircase path. \\
In other words, the dimension of $\ker(N^T)_0$ is precisely the number of upper diagonals
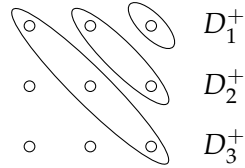
\begin{center}
\begin{tikzpicture}[scale=0.8]
    \foreach \x in {1,2,3}
    \foreach \y in {1,2,3}
    {
    \node[vertex] at (\x,\y) {};
    }
\draw[rotate=45] (3.525,0) ellipse (0.25cm and 1.1cm);
\draw[rotate=45] (4.25,0) ellipse (0.2cm and 0.5cm);
\draw[rotate=45] (2.825,0) ellipse (0.3cm and 1.8cm);
\node at (4.2,3) {$D_1^+$};
\node at (4.2,2) {$D_2^+$};
\node at (4.2,1) {$D_3^+$};
  \end{tikzpicture}
\end{center}
which lie completely above the staircase path in $\Alm$, which by definition is $\alpha(\unl,\um)$.
\end{example}

This example will show the statement $\cdim(\mathrm{im}(M^T)_0)= k -\beta(\unl,\um)$ in Lemma \ref{cor: image dimension} in the $k=3$ case.

\begin{example}[$k=3$ example] \label{ex: dim im M}
Applying $M^T$ to a generic element $\underline{b}$ of $\mathbf{J}(1) = J(3) \oplus J(2-i_1) \oplus J(1-i_2)$ gives
\begin{align*}
M^T(\underline{b})&= 
\begin{pmatrix}
b_{13}x^2y^{-j_{3}-3}\\
(b_{12}+b_{23})x^2y^{i_1-j_3-2} + b_{13}xy^{i_1-j_{3}-3}\\
(b_{11}+b_{22}+b_{33})x^2y^{i_2-j_3-1} + (b_{12}+b_{23})xy^{i_2-j_{3}-2}- b_{13}y^{i_2-j_{3}-3}
\end{pmatrix}\\
&= b_{13}\begin{pmatrix}
x^2y^{-j_{3}-3}\\
xy^{i_1-j_{3}-3}\\
y^{i_2-j_{3}-3}
\end{pmatrix}
+ (b_{12}+b_{23}) \begin{pmatrix}
0\\
x^2y^{i_1-j_3-2}\\
xy^{i_2-j_{3}-2}
\end{pmatrix}
+ (b_{11}+b_{22}+b_{33})
\begin{pmatrix}
0\\
0\\
x^2y^{i_2-j_3-1}
\end{pmatrix}.
\end{align*}
From this, we see that the dimension of $\mathrm{im}(M^T)_0$ is at most three, with possible basis vectors:
\begin{align*}
\begin{pmatrix} 
x^2y^{-j_3-3}\\
xy^{i_1-j_3-3}\\
y^{i_2-j_3-3}
\end{pmatrix} \in \mathrm{im}(M^T)_0 &\iff \text{$b_{13}$ can be nonzero}\\
&\iff  \ell_3 < m_1
\end{align*}
which by definition is if and only if, in $\Blm$, at least one of the vertices in the diagonal $D_1^-$, circled below,  
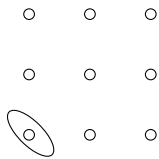
\begin{center}
\begin{tikzpicture}[scale=0.8]
    \foreach \x in {1,2,3}
    \foreach \y in {1,2,3}
    {
    \node[vertex] at (\x,\y) {};
    }
\draw[rotate=45] (1.45,0) ellipse (0.2cm and 0.5cm);
  \end{tikzpicture}
\end{center}
lies above the corresponding staircase path. Similarly,
\begin{align*}
\begin{pmatrix} 
0\\
x^2y^{i_1-j_3-2}\\
xy^{i_2-j_3-2}
\end{pmatrix} \in \mathrm{im}(M^T)_0 &\iff \text{at least one of $b_{12}, b_{23}$ can be nonzero}\\
&\iff  \ell_3 < m_2 \ \text{or} \ \ell_2 < m_1
\end{align*}
which by definition is if and only if, in $\Blm$, at least one of the vertices in the diagonal $D_2^-$, circled below,  
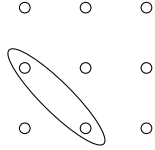
\begin{center}
\begin{tikzpicture}[scale=0.8]
    \foreach \x in {1,2,3}
    \foreach \y in {1,2,3}
    {
    \node[vertex] at (\x,\y) {};
    }
\draw[rotate=45] (2.15,0) ellipse (0.25cm and 1.1cm);
  \end{tikzpicture}
\end{center}
lies above the corresponding staircase path. \\
And finally,
\begin{align*}
\begin{pmatrix} 
0\\
0\\
x^2y^{i_2-j_3-1}
\end{pmatrix} \in \mathrm{im}(M^T)_0 &\iff \text{at least one of $b_{11}, b_{22}, b_{33}$ can be nonzero}\\
&\iff  \ell_3 < m_3 \ \text{or} \ \ell_2 < m_2  \ \text{or} \ \ell_1 < m_1
\end{align*}
which by definition is if and only if, in $\Blm$, at least one of the vertices in the diagonal $D_3^-$, circled below,  
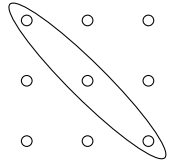
\begin{center}
\begin{tikzpicture}[scale=0.8]
    \foreach \x in {1,2,3}
    \foreach \y in {1,2,3}
    {
    \node[vertex] at (\x,\y) {};
    }
\draw[rotate=45] (2.825,0) ellipse (0.3cm and 1.8cm);
  \end{tikzpicture}
\end{center}
lies above the corresponding staircase path. In other words, the dimension of $\mathrm{im}(M^T)_0$ is the number of the circled diagonals in
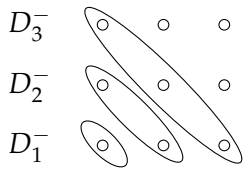
\begin{center}
\begin{tikzpicture}[scale=0.8]
    \foreach \x in {1,2,3}
    \foreach \y in {1,2,3}
    {
    \node[vertex] at (\x,\y) {};
    }
\draw[rotate=45] (2.825,0) ellipse (0.3cm and 1.8cm);
\draw[rotate=45] (2.15,0) ellipse (0.25cm and 1.1cm);
\draw[rotate=45] (1.45,0) ellipse (0.2cm and 0.5cm);
\node at (-0.2,3) {$D_3^-$};
\node at (-0.2,2) {$D_2^-$};
\node at (-0.2,1) {$D_1^-$};
  \end{tikzpicture}
\end{center}
which lie partially above the staircase path in $\Blm$. Or equivalently, the dimension of $\mathrm{im}(M^T)_0$ is $3$ minus the number of lower diagonals which lie completely below the the staircase path in $\Blm$, which by definition is $\beta(\unl,\um)$. Hence $\cdim(\mathrm{im}(M^T)_0)=3-\beta(\unl,\um)$.
\end{example}

\begin{example}
Returning to Example \ref{Ex:bfJandstaircase}, we compute $\cdim(\Ext^1(I,J))$ for the  graded ideals $I=(x^2,xy,y^2)$ and $J=(x^2,x,y^2)(-1)$ of $R=\mathbb{C}[x,y]/(x^3)$. Recall that these ideals correspond to the $3$-subsets $\unl=(-2,0,2)$ and $\um=(-1,2,3)$, and we may compute that $\alpha(\unl,\um)=3$, $\beta(\unl,\um)=2$ and $|\unl \cap \um |=1$. Using Theorem \ref{T: ext dimension}, this shows that
\[
\cdim(\Ext^1(I,J))=3+2-3-1=1.
\]
We see here that $\cdim(\Ext^1(I,J)) \neq 0$ which coincides with the fact that there is a crossing
\[
\ell_1 < m_1 < \ell_2 < m_3.
\]
\end{example}

\subsection{Reduction to Disjoint Case} \label{S:reduction}
Return now to the general setting of $k \geq 2$. The dimension formula for $\Ext^1$ given in Theorem \ref{T: ext dimension} allowed us to directly prove Theorem \ref{T:compatible} in the case where $\unl$ and $\um$ are disjoint. In this final section, we complete the proof of Theorem \ref{T:compatible} by showing that when $\unl$ and $\um$ are not disjoint, we may reduce the problem to a setting where they are.

Suppose that $\unl$ and $\um$ are $k$-subsets such that $\unl \cap \um$ is nonempty. In particular, suppose that we have a pair $(i,j)$ such that $\ell_i=m_j$. Note that this corresponds to a difference between $\Alm$ and $\Blm$; $\Alm_{i,j}$ will be shaded but $\Blm_{i,j}$ will not be. We may form two new $(k-1)$-subsets $\tunl$ and $\tum$ by deleting $\ell_i=m_j$ from $\unl$ and $\um$ respectively:
\begin{align*}
\widetilde{\ell}_p = \begin{cases}
               \ell_p               & \text{if $1 < p< i$} \\
               \ell_{p+1}      & \text{if $i \leq p \leq k-1$}
           \end{cases} \quad \text{and} \quad
\widetilde{m}_q = \begin{cases}
               m_q               & \text{if $1 < q< j$} \\
               m_{q+1}      & \text{if $j \leq q \leq k-1$.}
           \end{cases}
\end{align*}
\begin{example} \label{ex:nondisjoint}
($k=5$) Taking $\unl$ and $\um$ with 
\begin{align*}
\ell_1< m_1 < \ell_2=m_2 < \ell_3 < m_3 < m_4< \ell_4 < m_5 < \ell_5
\end{align*}
and removing $\ell_2=m_2$ gives $\tunl$ and $\tum$ satisfying:
\begin{align*}
\widetilde{\ell}_1< \widetilde{m}_1 < \widetilde{\ell}_2 < \widetilde{m}_2 < \widetilde{m}_3< \widetilde{\ell}_3 < \widetilde{m}_4 < \widetilde{\ell}_4
\end{align*}
which have $\Alm$ and $\tAlm$ as follows:
\begin{center}
  \begin{tikzpicture}[scale=0.67]
    \foreach \x in {1,2,3,4,5}
    \foreach \y in {1,2,3,4,5}
    {
    \node[vertex] at (\x,\y) {};
    }
    \fill (1,6-1) node[fvertex] {};
    \fill (2,6-1) node[fvertex] {};
    \fill (3,6-1) node[fvertex] {};
    \fill (4,6-1) node[fvertex] {};
    \fill (5,6-1) node[fvertex] {};
    \fill (2,6-2) node[fvertex] {};
    \fill (3,6-2) node[fvertex] {};
    \fill (4,6-2) node[fvertex] {};
    \fill (5,6-2) node[fvertex] {};
    \fill (3,6-3) node[fvertex] {};
    \fill (4,6-3) node[fvertex] {};
    \fill (5,6-3) node[fvertex] {};
    \fill (5,6-4) node[fvertex] {};
 \draw (0.5,4.5) -- (5.5,4.5);
\draw (0.5,3.5)--(5.5,3.5);
\draw (2.5,0.5)--(2.5,5.5);
\draw (1.5,0.5)--(1.5,5.5);
 \draw[draw=green, fill=yellow, opacity=0.2] (0.5,5.5) -- (1.5,5.5)--(1.5,4.5)--(0.5,4.5) -- cycle;
 \draw[draw=green, fill=yellow, opacity=0.2] (2.5,5.5) -- (5.5,5.5)--(5.5,4.5)--(2.5,4.5) -- cycle;
 \draw[draw=green, fill=yellow, opacity=0.2] (0.5,3.5) -- (1.5,3.5)--(1.5,0.5)--(0.5,0.5) -- cycle;
 \draw[draw=green, fill=yellow, opacity=0.2] (2.5,3.5) -- (5.5,3.5)--(5.5,0.5)--(2.5,0.5) -- cycle;
  \end{tikzpicture}
\hspace{1cm}
  \begin{tikzpicture}[scale=0.67]
    \foreach \x in {1,2,3,4}
    \foreach \y in {1,2,3,4}
    {
    \node[vertex] at (\x,\y) {};
    }
    \fill (1,5-1) node[fvertex] {};
    \fill (2,5-1) node[fvertex] {};
    \fill (3,5-1) node[fvertex] {};
    \fill (4,5-1) node[fvertex] {};
    \fill (2,5-2) node[fvertex] {};
    \fill (3,5-2) node[fvertex] {};
    \fill (4,5-2) node[fvertex] {};
    \fill (4,5-3) node[fvertex] {};
\draw (0.5,3.5)--(4.5,3.5);
\draw (1.5,0.5)--(1.5,4.5);
 \draw[draw=yellow, fill=yellow, opacity=0.2] (0.5,4.5) -- (4.5,4.5)--(4.5,0.5)--(0.5,0.5) -- cycle;
  \end{tikzpicture}
\end{center}
\end{example}

\begin{lemma}
Given $k$-subsets $\unl$ and $\um$ as above, $\tAlm$ is obtained from $\Alm$ be deleting row $i$ and column $j$. Analogously, $\tBlm$ is obtained from $\Blm$ be deleting row $i$ and column $j$.
\end{lemma}
\begin{proof}
By deleting row $i$ and column $j$ in $\Alm$ we split $\Alm$ into (up to) four regions:
\begin{itemize}
\item $\Alm_{p,q}$ where $1 \leq p < i$ and $1 \leq q < j$;
\item $\Alm_{p,q}$ where $1 \leq p < i$ and $j < q \leq k$;
\item $\Alm_{p,q}$ where $i < p \leq k $ and $1 \leq q < j$;
\item $\Alm_{p,q}$ where $i < p \leq k$ and $j < q \leq k$.
\end{itemize}
In the first case, we wish to identify $\Alm_{p,q}$ with $\tAlm_{p,q}$. In this region, we have $\ell_p=\widetilde{\ell}_p$ and $m_q=\widetilde{m}_q$ and hence 
\begin{align*}
\ell_p \leq m_q \iff \widetilde{\ell}_p \leq \widetilde{m}_q
\end{align*}
or in other words, $\Alm_{p,q}$ is filled if and only if $\tAlm_{p,q}$ is filled, as required.\\
In the second region, we wish to identify $\Alm_{p,q}$ with $\tAlm_{p,q-1}$. In this region, we have we have $\ell_p=\widetilde{\ell}_p$ but since $q-1 \geq j$, we also have $m_{q}=\widetilde{m}_{q-1}$. Thus,
\begin{align*}
\ell_p \leq m_q \iff \widetilde{\ell}_p \leq \widetilde{m}_{q-1}
\end{align*}
or in other words, $\Alm_{p,q}$ is filled if and only if $\tAlm_{p,q-1}$ is filled, as required.\\
In the third region, we wish to identify $\Alm_{p,q}$ with $\tAlm_{p-1,q}$. In this region, we have we have $m_q=\widetilde{m}_q$ but since $p-1 \geq i$, we also have $\ell_{p}=\widetilde{\ell}_{p-1}$. Thus,
\begin{align*}
\ell_p \leq m_q \iff \widetilde{\ell}_{p-1} \leq \widetilde{m}_{q}
\end{align*}
or in other words, $\Alm_{p,q}$ is filled if and only if $\tAlm_{p-1,q}$ is filled, as required.\\
In the final region, we wish to identify $\Alm_{p,q}$ with $\tAlm_{p-1,q-1}$. In this region, we have we have $\ell_p=\widetilde{\ell}_{p-1}$ and $m_{q}=\widetilde{m}_{q-1}$. Thus,
\begin{align*}
\ell_p \leq m_q \iff \widetilde{\ell}_{p-1} \leq \widetilde{m}_{q-1}
\end{align*}
or in other words, $\Alm_{p,q}$ is filled if and only if $\tAlm_{p-1,q-1}$ is filled, as required.\\
The proof for $\Blm$ is exactly the same with all the inequalities changed to strict inequalities.
\end{proof}

\begin{lemma} \label{lem: diagonals}
Suppose $\unl$ and $\um$ are $k$-subsets with $\ell_i=m_j$. Then $j-i \leq k-\alpha(\unl,\um)$ and $i-j \leq k-\beta(\unl,\um)$. 
\end{lemma}
\begin{proof}
Suppose for a contradiction that $j-i > k-\alpha(\unl,\um)$. Since $\alpha(\unl,\um) \leq k$, this implies $j-i > 0$ and hence we have $j >1$. Then,
\begin{align*}
j-i > k-\alpha(\unl,\um)
&\iff (j-i) -1 > (k-\alpha(\unl,\um))-1\\
&\iff (j-1)-i \geq k-\alpha(\unl,\um).
\end{align*}
This shows $(i,j-1)$ lies on the diagonal $D_p^+$ for some $p \leq \alpha(\unl,\um)$ and so, by definition of $\alpha(\unl,\um)$, the final line implies $\Alm_{i,j-1}$ is shaded. This implies $\ell_i \leq m_{j-1}$ which further implies $m_j=\ell_i \leq m_{j-1}$ which is a contradiction. The proof for $i-j \leq k-\beta(\unl,\um)$ is similar.
\end{proof}

\begin{lemma} \label{lem: ca geq}
With the setup above $\alpha(\tunl,\tum) \geq \alpha(\unl,\um)-1$.
\end{lemma}

\begin{proof}
First note that since $\ell_i=m_j$, there is at least one vertex above the staircase path in $\Alm$ and hence $\alpha(\unl,\um) > 0$. Also, by the definition of $\alpha(\unl,\um)$, we know that for all $s \leq \alpha(\unl,\um)$ the diagonal $D_s^+$ is completely shaded, or equivalently,
\begin{align} \label{inequality in A}
\text{for all pairs $(p,q)$ with $1 \leq p,q \leq k$ and $q-p \geq k-\alpha(\unl,\um)$, we have $\ell_p \leq m_q$}.
\end{align}
We will prove that in $\tAlm$, the diagonal $D_{\alpha(\unl,\um)-1}^+$ lies completely above the staircase path from which the result follows. Take a pair $(p,q)$ on this diagonal i.e.\ with $1 \leq p,q \leq k-1$ and $q-p=k-\alpha(\unl,\um)=(k-1)-(\alpha(\unl,\um)-1)$. Using \eqref{inequality in A} we see that
\begin{align}
\ell_p \leq m_q \quad \text{and} \quad \ell_{p+1} \leq m_{q+1} \label{inequalities 1}
\end{align}
where the latter holds since $1 \leq p+1, q+1 \leq k$ and $(q+1)-(p+1)=q-p \geq k-\alpha(\unl,\um)$.\\
Now, the pair $(p,q)$ must lie in one of four regions:
\begin{enumerate}
\item $1 \leq p < i$ and $1 \leq q < j$;
\item $i \leq p \leq k-1$ and $1 \leq q < j$;
\item $1 \leq p < i$ and $j \leq q \leq k-1$;
\item $i \leq p \leq k-1$ and $j \leq q \leq k-1$.
\end{enumerate}
If $(p,q)$ lies in the first region, then $\ell_p=\widetilde{\ell}_{p}$ and $m_q=\widetilde{m}_{q}$. Then, using \eqref{inequalities 1},
\begin{align*}
\widetilde{\ell}_{p}=\ell_p \leq m_q = \widetilde{m}_{q}.
\end{align*}
If $(p,q)$ lies in the second region, then, using Lemma \ref{lem: diagonals}, 
\begin{align*}
q-p < j-i \leq k-\alpha(\unl,\um)
\end{align*}
and thus no such $(p,q)$ lies on the diagonal $D_{\alpha(\unl,\um)-1}$ in $\tAlm$.\\
If $(p,q)$ lies in the third region, then $\ell_p=\widetilde{\ell}_{p}$ and $m_{q+1}=\widetilde{m}_{q}$. Then, using \eqref{inequalities 1}, and that $m_q < m_{q+1}$ shows that
\begin{align*}
\widetilde{\ell}_{p}=\ell_p \leq m_q < m_{q+1} = \widetilde{m}_{q}.
\end{align*}
If $(p,q)$ lies in the fourth region, then $\ell_{p+1}=\widetilde{\ell}_{p}$ and $m_{q+1}=\widetilde{m}_{q}$. Then, using \eqref{inequalities 1},
\begin{align*}
\widetilde{\ell}_{p}=\ell_{p+1} \leq m_{q+1} = \widetilde{m}_{q}.
\end{align*}
Thus we have shown that in $\tAlm$, all $(p,q)$ on the diagonal $D_{\alpha(\unl,\um)-1}^+$ lie above the staircase path as required.

\end{proof}
\begin{lemma} \label{lem: ca leq} 
With the setup above $\alpha(\tunl,\tum) \leq \alpha(\unl,\um)-1$.
\end{lemma}
\begin{proof}
Suppose that $\alpha(\unl,\um)=k$. Since $\tunl$ and $\tum$ are $(k-1)$-subsets, by definition, $\alpha(\tunl,\tum) \leq k-1$ and hence $\alpha(\tunl,\tum) \leq \alpha(\unl,\um)-1$ is trivial in this case.\\
Now suppose $\alpha(\unl,\um) < k$. Since $\alpha(\unl,\um)$ is maximal, there exists $(p,q) \in D_{\alpha(\unl,\um)+1}^+$ (i.e.\ $1 \leq p,q \leq k$ and $q-p=k-\alpha(\unl,\um)-1$) 
such that $\Alm_{p,q}$ lies below the staircase path or equivalently, such that $\ell_p > m_q$.\\
Since $\Alm_{i,j}$ lies above the staircase (as $\ell_i=m_j$), Lemma \ref{lem: regions above} shows that $\Alm_{s,t}$ lies above the path whenever we have both $s \leq i$ and $t \geq j$. Thus we must have $p >i$ or $q <j$. \\
Suppose $p>i$, or equivalently $p\geq i+1$. Then,
\begin{align*}
q&= p+k-\alpha(\unl,\um)-1\\
 &\geq (i+1) +(k-\alpha(\unl,\um)) -1 \tag{since $ p \geq i+1 $}\\
&\geq (i+1)+(j-i) -1 \tag{by Lemma \ref{lem: diagonals} }\\
&= j.
\end{align*}
So we have $p-1 \geq i$ and $ q \geq j$. Thus $\widetilde{\ell}_{p-1} = \ell_p$ and either
\begin{itemize}
\item $q=j$, and then $q-1<j$ and so $\widetilde{m}_{q-1} = m_{q-1} \leq m_q$;
\item $q >j$, and then $\widetilde{m}_{q-1} = m_{q}$.
\end{itemize}
Hence we have $\widetilde{\ell}_{p-1} = \ell_p$ and $\widetilde{m}_{q-1} \leq m_q$ and so,
\begin{align*}
\widetilde{\ell}_{p-1} = \ell_p > m_q \geq \widetilde{m}_{q-1}
\end{align*}
where the strict inequality holds as $\Alm_{p,q}$ lies below the staircase path. In particular, the pair $(p-1, q-1)$ satisfies
$(q-1)-(p-1)=q-p=k-\alpha(\unl,\um)-1$ and $\tAlm_{p-1,q-1}$ lies below the staircase. \\
If $q<j$, or equivalently $q \leq j-1$, then,
\begin{align*}
p &= q-k+\alpha(\unl,\um)+1\\
 &\leq (j-1) -(k-\alpha(\unl,\um)) +1\tag{since $ q \leq j-1$}\\
&\leq (j-1)-(j-i)+1 \tag{by Lemma \ref{lem: diagonals} }\\
&= i.
\end{align*}
So we have $q < j$ and $ p\leq i$. Thus $\widetilde{m}_{q} = m_q$ and either
\begin{itemize}
\item $p=i$, and then $\widetilde{\ell}_{p} = \ell_{p+1} > \ell_p$;
\item $p <i $, and then $\widetilde{\ell}_{p} = \ell_{p}$.
\end{itemize}
Hence we have $\widetilde{\ell}_{p} \geq \ell_p$ and $\widetilde{m}_{q} = m_q$ and so,
\begin{align*}
\widetilde{\ell}_{p} \geq \ell_p > m_q = \widetilde{m}_{q}
\end{align*}
where the strict inequality holds as $\Alm_{p,q}$ lies below the staircase path. In particular, the pair $(p, q)$ satisfy
$q-p=k-\alpha(\unl,\um)-1$ and $\tAlm_{p,q}$ lies below the staircase. Moreover, if $p=k$, then since we know $q< j \leq k$, we have $k-\alpha(\unl,\um)-1=q-p < 0$ contradicting $\alpha(\unl,\um) <k$. Thus we must have $1 \leq p,q \leq k-1$.\\

Thus we have shown that there exists $1 \leq p,q \leq k-1$ such that $q-p=(k-1)-\alpha(\unl,\um)$ and $\tAlm_{p,q}$ lies below the staircase. Equivalently, we have shown that in $\tAlm$, the diagonal $D_{\alpha(\unl,\um)}^+$ does not lie completely above the staircase path and thus $\alpha(\tunl,\tum) \leq \alpha(\unl,\um)-1$.

\end{proof}

\begin{corollary} \label{lem: ca drop}
With the setup above $\alpha(\tunl,\tum) = \alpha(\unl,\um)-1$.
\end{corollary}
\begin{proof}
Combine Lemmas \ref{lem: ca geq} and \ref{lem: ca leq}.
\end{proof}

\begin{corollary} \label{lem: cb drop}
With the setup above $\beta(\tunl,\tum) = \beta(\unl,\um)-1$.
\end{corollary}
\begin{proof}
    Combine Corollary \ref{lem: ca drop} and Lemma \ref{L:symmetry}.
\end{proof}

\begin{example}
Continuing Example \ref{ex:nondisjoint} we see that $\alpha(\unl,\um)=4$ and $\alpha(\tunl, \tum)=3$:
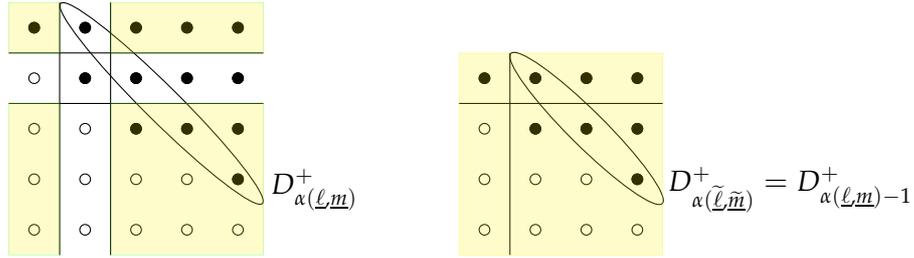
\begin{center}
\hspace{1.5cm}
  \begin{tikzpicture}[scale=0.67]
  \draw[rotate=45] (4.95,0) ellipse (0.35cm and 2.8cm);
  \node at (6.5,1.8) {$D^+_{\alpha(\unl,\um)}$};
    \foreach \x in {1,2,3,4,5}
    \foreach \y in {1,2,3,4,5}
    {
    \node[vertex] at (\x,\y) {};
    }
    \fill (1,6-1) node[fvertex] {};
    \fill (2,6-1) node[fvertex] {};
    \fill (3,6-1) node[fvertex] {};
    \fill (4,6-1) node[fvertex] {};
    \fill (5,6-1) node[fvertex] {};
    \fill (2,6-2) node[fvertex] {};
    \fill (3,6-2) node[fvertex] {};
    \fill (4,6-2) node[fvertex] {};
    \fill (5,6-2) node[fvertex] {};
    \fill (3,6-3) node[fvertex] {};
    \fill (4,6-3) node[fvertex] {};
    \fill (5,6-3) node[fvertex] {};
    \fill (5,6-4) node[fvertex] {};
 \draw (0.5,4.5) -- (5.5,4.5);
\draw (0.5,3.5)--(5.5,3.5);
\draw (2.5,0.5)--(2.5,5.5);
\draw (1.5,0.5)--(1.5,5.5);
 \draw[draw=green, fill=yellow, opacity=0.2] (0.5,5.5) -- (1.5,5.5)--(1.5,4.5)--(0.5,4.5) -- cycle;
 \draw[draw=green, fill=yellow, opacity=0.2] (2.5,5.5) -- (5.5,5.5)--(5.5,4.5)--(2.5,4.5) -- cycle;
 \draw[draw=green, fill=yellow, opacity=0.2] (0.5,3.5) -- (1.5,3.5)--(1.5,0.5)--(0.5,0.5) -- cycle;
 \draw[draw=green, fill=yellow, opacity=0.2] (2.5,3.5) -- (5.5,3.5)--(5.5,0.5)--(2.5,0.5) -- cycle;
  \end{tikzpicture}
\hspace{1cm}
  \begin{tikzpicture}[scale=0.67]
    \draw[rotate=45] (4.25,0) ellipse (0.35cm and 2.1cm);
  \node at (7,1.8) {$D^+_{\alpha(\tunl,\tum)}=D^+_{\alpha(\unl,\um)-1}$};
    \foreach \x in {1,2,3,4}
    \foreach \y in {1,2,3,4}
    {
    \node[vertex] at (\x,\y) {};
    }
    \fill (1,5-1) node[fvertex] {};
    \fill (2,5-1) node[fvertex] {};
    \fill (3,5-1) node[fvertex] {};
    \fill (4,5-1) node[fvertex] {};
    \fill (2,5-2) node[fvertex] {};
    \fill (3,5-2) node[fvertex] {};
    \fill (4,5-2) node[fvertex] {};
    \fill (4,5-3) node[fvertex] {};
\draw (0.5,3.5)--(4.5,3.5);
\draw (1.5,0.5)--(1.5,4.5);
 \draw[draw=yellow, fill=yellow, opacity=0.2] (0.5,4.5) -- (4.5,4.5)--(4.5,0.5)--(0.5,0.5) -- cycle;
  \end{tikzpicture}
\end{center}

\end{example}

\begin{corollary} \label{cor:disjoint case}
A pair of $k$-subsets $\unl$ and $\um$ are noncrossing if and only if
\[
\alpha(\unl,\um)+\beta(\unl,\um) -|\unl \cap \um|=k.
\]
\end{corollary}

\begin{proof}
Suppose that $\unl$ and $\um$ are disjoint. Then the result is precisely Lemma \ref{lem: single step} and we are done.

If $\unl$ and $\um$ are not disjoint, suppose that $\ell_i=m_j$. Then, consider the $k-1$-subsets $\tunl \colonequals \unl \setminus \{\ell_i\}$ and $\tum\colonequals \um \setminus \{m_j\}$. Then, 
\begin{itemize}
\item $\tunl$ and $\tum$ are noncrossing if and only if $\unl$ and $\um$ are noncrossing;
\item $|\tunl \cap \tum|=|\unl \cap \um|-1$;
\item $\alpha(\tunl,\tum)=\alpha(\unl,\um)-1$ using Corollary \ref{lem: ca drop};
\item $\beta(\tunl,\tum)=\beta(\unl,\um)-1$ using Corollary \ref{lem: cb drop};
\end{itemize}
If $t \colonequals |\unl \cap \um| =1$ then $\tunl$ and $\tum$ are disjoint. If not repeat the process by removing another equality, and continue until you end up with disjoint $(k-t)$-subsets $\tunl$ and $\tum$ such that:
\begin{itemize}
\item $\tunl$ and $\tum$ are noncrossing if and only if $\unl$ and $\um$ are noncrossing;
\item $\alpha(\tunl,\tum)=\alpha(\unl,\um)-t$;
\item $\beta(\tunl,\tum)=\beta(\unl,\um)-t$;
\end{itemize}
Then,
\begin{align*}
\alpha(\unl,\um) + \beta(\unl,\um)-|\unl\cap\um|=k &\iff \alpha(\unl,\um) + \beta(\unl,\um)-t=k \\
&\iff (\alpha(\unl,\um)-t)+(\beta(\unl,\um)-t) =k-t \tag{rearrange}\\
&\iff \alpha(\tunl,\tum) + \beta(\tunl,\tum) =k-t \\
&\iff \text{$\tunl$ and $\tum$ are noncrossing} \tag{by Lemma \ref{lem: single step}}\\
&\iff \text{$\unl$ and $\um$ are noncrossing}.
\end{align*}
\end{proof}

We are now ready to prove Theorem \ref{T:compatible}, keeping the notation from the beginning of Section \ref{S:tool}, where $\unl = \unl(I)$ and $\um = \unl(J)$.

\begin{proof}[Proof of Theorem \ref{T:compatible}]
Recall that we wish to show $\cdim(\Ext^1(I,J))=0$ if and only if $\unl$ and $\um$ are noncrossing. But this follows directly from Theorem \ref{T: ext dimension} and Corollary \ref{cor:disjoint case}.
\end{proof}


\newcommand{\etalchar}[1]{$^{#1}$}

\end{document}